\pdfoutput=1
\newif\ifpersonal
\RequirePackage[l2tabu,orthodox]{nag} 
\documentclass[12pt]{amsart} 
\usepackage{amsmath,amsthm,amssymb,mathrsfs,mathtools,bm,eucal,tensor} 
\usepackage{microtype,fixltx2e,lmodern} 
\usepackage[utf8]{inputenc} 
\usepackage[T1]{fontenc} 
\usepackage{enumerate,comment,braket,xspace,tikz-cd,csquotes} 
\usepackage[centering,vscale=0.7,hscale=0.8]{geometry}
\usepackage[hidelinks,linktoc=all]{hyperref}
\usepackage[capitalize]{cleveref}

\numberwithin{equation}{section}
\theoremstyle{plain}
\newtheorem{thm-intro}{Theorem}[section]
\newtheorem{thm}{Theorem}[section]
\newtheorem{lem}[thm]{Lemma}
\newtheorem{prop}[thm]{Proposition}

\newtheorem{cor}[thm]{Corollary}

\theoremstyle{definition}
\newtheorem{defin}[thm]{Definition}
\newtheorem{notation}[thm]{Notation}
\newtheorem{eg}[thm]{Example}
\newtheorem{rem}[thm]{Remark}

\ifpersonal
\newcommand*{\personal}[1]{\textcolor[rgb]{0,0,1}{(Personal: #1)}}
\newcommand*{\todo}[1]{\textcolor{red}{(Todo: #1)}}
\else
\newcommand*{\personal}[1]{\ignorespaces}
\newcommand*{\todo}[1]{\ignorespaces}
\fi

\newcommand{\C}{\mathbb C}

\newcommand{\rH}{\mathrm H}

\newcommand{\rR}{\mathrm R}

\newcommand{\cC}{\mathcal C}
\newcommand{\cD}{\mathcal D}

\newcommand{\cF}{\mathcal F}
\newcommand{\cH}{\mathcal H}
\newcommand{\cG}{\mathcal G}

\newcommand{\cO}{\mathcal O}

\newcommand{\cS}{\mathcal S}
\newcommand{\cT}{\mathcal T}

\newcommand{\cX}{\mathcal X}
\newcommand{\cY}{\mathcal Y}

\DeclareFontFamily{U}{BOONDOX-calo}{\skewchar\font=45 }
\DeclareFontShape{U}{BOONDOX-calo}{m}{n}{<-> s*[1.05] BOONDOX-r-calo}{}
\DeclareFontShape{U}{BOONDOX-calo}{b}{n}{<-> s*[1.05] BOONDOX-b-calo}{}
\DeclareMathAlphabet{\mathcalboondox}{U}{BOONDOX-calo}{m}{n}

\newcommand{\bbA}{\mathbb A}

\newcommand{\bA}{\mathbf A}

\newcommand{\bP}{\mathbf P}

\makeatletter
\let\save@mathaccent\mathaccent
\newcommand*\if@single[3]{%
	\setbox0\hbox{${\mathaccent"0362{#1}}^H$}%
	\setbox2\hbox{${\mathaccent"0362{\kern0pt#1}}^H$}%
	\ifdim\ht0=\ht2 #3\else #2\fi
}
\newcommand*\rel@kern[1]{\kern#1\dimexpr\macc@kerna}
\newcommand*\widebar[1]{\@ifnextchar^{{\wide@bar{#1}{0}}}{\wide@bar{#1}{1}}}
\newcommand*\wide@bar[2]{\if@single{#1}{\wide@bar@{#1}{#2}{1}}{\wide@bar@{#1}{#2}{2}}}
\newcommand*\wide@bar@[3]{%
	\begingroup
	\def\mathaccent##1##2{%
		\let\mathaccent\save@mathaccent
		\if#32 \let\macc@nucleus\first@char \fi
		\setbox\z@\hbox{$\macc@style{\macc@nucleus}_{}$}%
		\setbox\tw@\hbox{$\macc@style{\macc@nucleus}{}_{}$}%
		\dimen@\wd\tw@
		\advance\dimen@-\wd\z@
		\divide\dimen@ 3
		\@tempdima\wd\tw@
		\advance\@tempdima-\scriptspace
		\divide\@tempdima 10
		\advance\dimen@-\@tempdima
		\ifdim\dimen@>\z@ \dimen@0pt\fi
		\rel@kern{0.6}\kern-\dimen@
		\if#31
		\overline{\rel@kern{-0.6}\kern\dimen@\macc@nucleus\rel@kern{0.4}\kern\dimen@}%
		\advance\dimen@0.4\dimexpr\macc@kerna
		\let\final@kern#2%
		\ifdim\dimen@<\z@ \let\final@kern1\fi
		\if\final@kern1 \kern-\dimen@\fi
		\else
		\overline{\rel@kern{-0.6}\kern\dimen@#1}%
		\fi
	}%
	\macc@depth\@ne
	\let\math@bgroup\@empty \let\math@egroup\macc@set@skewchar
	\mathsurround\z@ \frozen@everymath{\mathgroup\macc@group\relax}%
	\macc@set@skewchar\relax
	\let\mathaccentV\macc@nested@a
	\if#31
	\macc@nested@a\relax111{#1}%
	\else
	\def\gobble@till@marker##1\endmarker{}%
	\futurelet\first@char\gobble@till@marker#1\endmarker
	\ifcat\noexpand\first@char A\else
	\def\first@char{}%
	\fi
	\macc@nested@a\relax111{\first@char}%
	\fi
	\endgroup
}
\makeatother

\newcommand{\PSh}{\mathrm{PSh}}
\newcommand{\Sh}{\mathrm{Sh}}

\newcommand{\infcat}{$\infty$-category\xspace}

\newcommand{\inftopos}{$\infty$-topos\xspace}
\newcommand{\inftopoi}{$\infty$-topoi\xspace}

\newcommand{\rSet}{\mathrm{Set}}
\newcommand{\Ab}{\mathrm{Ab}}
\newcommand{\DAb}{{\cD(\Ab)}}
\newcommand{\tauan}{\tau_\mathrm{an}}

\newcommand{\Mod}{\textrm{-}\mathrm{Mod}}

\newcommand{\Coh}{\mathrm{Coh}}
\newcommand{\Cohb}{\mathrm{Coh}^\mathrm{b}}
\newcommand{\Cohh}{\mathrm{Coh}^\heartsuit}

\newcommand{\St}{\mathrm{St}}
\newcommand{\Stn}{\mathrm{Stn}_{\mathbb C}}
\newcommand{\Sch}{\mathrm{Sch}}

\newcommand{\Top}{\mathcal T\mathrm{op}}

\newcommand{\dAnc}{\mathrm{dAn}_{\mathbb C}}

\newcommand{\cTan}{\cT_{\mathrm{an}}}
\newcommand{\cTank}{\cT_{\mathrm{an}}(k)}
\newcommand{\cTdisc}{\cT_{\mathrm{disc}}}

\newcommand{\cTet}{\cT_{\mathrm{\acute{e}t}}}

\newcommand{\Str}{\mathrm{Str}}
\newcommand{\Strloc}{\mathrm{Str}^\mathrm{loc}}
\newcommand{\RTop}{\tensor*[^\rR]{\Top}{}}

\newcommand{\dStn}{\mathrm{dStn}_{\mathbb C}}

\newcommand{\trunc}{\mathrm{t}_0}

\newcommand{\CRing}{\mathrm{CRing}}

\newcommand{\Anc}{\mathrm{An}_{\mathbb C}}

\newcommand{\fib}{\mathrm{fib}}

\newcommand{\anL}{\mathbb L\an}
\newcommand{\AnRing}{\mathrm{AnRing}}

\newcommand{\cHom}{\cH \mathrm{om}}

\newcommand{\llb}{[\![}
\newcommand{\rrb}{]\!]}

\newcommand{\an}{^\mathrm{an}}
\newcommand{\alg}{^\mathrm{alg}}

\newcommand{\et}{_\mathrm{\acute{e}t}}
\newcommand{\ev}{\mathrm{ev}}

\newcommand{\inv}{^{-1}}

\newcommand{\kanal}{$k$-analytic\xspace}

\newcommand{\op}{^\mathrm{op}}

\newcommand{\DM}{Deligne-Mumford\xspace}

\usetikzlibrary{decorations.markings} 
\tikzset{
  closed/.style = {decoration = {markings, mark = at position 0.5 with { \node[transform shape, xscale = .8, yscale=.4] {/}; } }, postaction = {decorate} },
  open/.style = {decoration = {markings, mark = at position 0.5 with { \node[transform shape, scale = .7] {$\circ$}; } }, postaction = {decorate} }
}

\DeclareMathOperator{\Fun}{Fun}

\DeclareMathOperator{\Hom}{Hom}

\DeclareMathOperator{\Map}{Map}

\DeclareMathOperator{\Sp}{Sp}

\DeclareMathOperator{\Spec}{Spec}

\DeclareMathOperator*{\colim}{colim}

\begin{document}
\title{GAGA theorems in derived complex geometry}

\author{Mauro PORTA}
\address{Mauro PORTA, Institut de Recherche Mathématique Avancée, 7 rue René Descartes, 67000 Strasbourg, France}
\email{porta@math.unistra.fr}

\date{\today}

\begin{abstract}
	In this paper, we expand the foundations of derived complex analytic geometry introduced in \cite{DAG-IX}.
	We start by studying the analytification functor and its properties. In particular, we prove that for a derived complex scheme locally almost of finite presentation $X$, the canonical map $X\an \to X$ is flat in the derived sense.
	Next, we provide a comparison result relating derived complex analytic spaces to geometric stacks.
	Using these results and building on the previous work \cite{Porta_Yu_Higher_analytic_stacks_2014}, we prove a derived version of the GAGA theorems.
	As an application, we prove that the infinitesimal deformation theory of a derived complex analytic moduli problem is governed by a differential graded Lie algebra.
\end{abstract}

\maketitle

\personal{PERSONAL COMMENTS ARE SHOWN!!!}

\tableofcontents

\section{Introduction} \label{sec:introduction}

In recent years derived analytic geometry has received much attention.
First introduced in \cite{DAG-IX} and in the author's Ph.D.\ thesis, applications to mirror symmetry \cite{Porta_Yu_DNAnG_I,Porta_Yu_Representability} and to (nonabelian) Hodge theory \cite{Di_Natale_Period,Porta_RH} started to appear.
Both these applications make an essential use of derived analytic geometry.

Concerning mirror symmetry, it is derived \emph{non-archimedean} analytic geometry that plays a central role.
After the pioneering work of M.\ Kontsevich and Y.\ Soibelman \cite{Kontsevich_Homological_2001}, it has become understood that non-archimedean enumerative geometry is deeply related to mirror symmetry.
Derived non-archimedean analytic geometry is needed to construct quasi-smooth derived enhancement of the non-archimedean moduli space of stable maps (introduced in the non-archimedean setting by T.\ Y.\ Yu in \cite{Yu_Gromov_2014}).

In the case of nonabelian Hodge theory, derived analytic geometry is needed because it allows on one side to use the usual analytic topology and on the other hand it gives a better control on deformation theory.
The possibility of using the analytic topology has been exploited in \cite{Porta_RH} to prove a derived version of Deligne's Riemann-Hilbert correspondence.
As sketched in \cite{Simpson_Geometricity_2009}, this is the first step in the construction of Deligne-Hitchin's twistor space and in the study of its deformation theory, which is conjecturally governed by mixed Hodge modules.

One of the direct applications of the results in this paper is that the local structure of an analytic moduli problem is governed by differential graded Lie algebras (see \cref{cor:analytic_FMP}).
This relies on derived versions of the GAGA theorems, that constitute the main results of this work.
In more precise terms, we prove that analytic formal moduli problems are equivalent to algebraic formal moduli problems.
This result plays therefore a crucial role in the nonabelian Hodge theory program introduced by C.\ Simpson in \cite{Simpson_Geometricity_2009}. \\

\paragraph{\textbf{Main results}}

The first result of this work is to provide an alternative description of derived analytic spaces in terms of the functor of points.
We refer to \cref{sec:review} for a review of the notion of derived analytic space.
For the moment, let us just recall that a derived analytic space is a pair $X = (\cX, \cO_X)$, where $\cX$ is an $\infty$-topos that is locally a topological space and $\cO_X$ is a sheaf of \emph{analytic rings} on $\cX$.
Informally speaking, an analytic ring is a simplicial commutative ring equipped with an \emph{axiomatic holomorphic functional calculus}.
This is made precise through the language of pregeometries introduced in \cite{DAG-V}.
In particular, every analytic ring $A$ has an underlying simplicial commutative ring, denoted $A\alg$.

Locally, every derived analytic space $X$ has an underlying (uniquely determined) classical analytic space.
We say that $X$ is a derived Stein space if it has an underlying classical analytic space which is Stein.
We can thus state the first main result:

\begin{thm-intro}[cf.\ {\cref{thm:functor_of_points}}] \label{thm-intro:functor_of_points}
	Let $\dAnc$ denote the $\infty$-category of derived analytic spaces and let $\dStn$ be the full subcategory spanned by derived Stein spaces.
	Then, there is a fully faithful embedding
	\[ \dAnc \hookrightarrow \St(\dStn, \tauan) , \]
	where $\tauan$ is the derived version of the usual analytic topology introduced in \cref{subsec:geometric_context}.
	Furthermore, this embedding restricts to an equivalence between the full subcategory of $\dAnc$ spanned by those derived analytic spaces $X = (\cX, \cO_X)$ such that $\cX$ is $n$-localic on one side, and $n$-truncated analytic \DM stacks on the other side.
\end{thm-intro}

In Sections \ref{sec:analytification} and \ref{sec:local_theory} we study the derived analytification functor.
We can summarize the main results as follows:

\begin{thm-intro}[cf.\ {\cref{thm:derived_analytification} and \cref{cor:analytification_flat}}]
	Let $X \coloneqq (\cX, \cO_X)$ be a derived \DM stack locally almost of finite presentation over $\mathbb C$.
	Then the derived analytification $X\an \coloneqq (\cX\an, \cO_{X\an}) \in \dAnc$ exists, and furthermore:
	\begin{enumerate}
		\item when $X$ is an underived scheme locally of finite presentation over $\mathbb C$, its derived analytification coincides with the classical one of Grothendieck -- Serre \cite[Expos\'e XII]{SGA1};
		\item the canonical map $(\cX\an, \cO_{X\an}\alg) \to (\cX, \cO_X)$ is flat in the derived sense.
	\end{enumerate}
\end{thm-intro}

The key step in the proof of this theorem is the careful study of the $\infty$-category of analytic rings that can be found in \cref{sec:local_theory}.
In this part, we need some general facts about pregeometries that did not appear in \cite{DAG-V}.
We collected these needed results in the appendix.

In Sections \ref{sec:Grauert} and \ref{sec:GAGA} we approach the main results of this paper.
\cref{thm-intro:functor_of_points} allows us to introduce the notion of derived analytic Artin stack.
We can therefore use the notion of proper map of underived analytic Artin stacks studied in \cite{Porta_Yu_Higher_analytic_stacks_2014} to formulate and prove the following derived versions of Grauert and GAGA theorems:

\begin{thm-intro}[cf.\ {Theorems \ref{thm:Grauert_theorem}, \ref{thm:GAGA1} and \ref{thm:GAGA2}}]
	Say that a morphism of derived (algebraic or analytic) Artin stacks $f \colon X \to Y$ is proper if and only if its truncation $\trunc(f)$ is proper.
	Then:
	\begin{enumerate}
		\item if $f \colon X \to Y$ is a proper morphism of derived analytic Artin stacks, then the derived pushforward induces a functor of stable $\infty$-categories (written in \emph{homological} notation):
		\[ f_* \colon \Coh^-(X) \to \Coh^-(Y) . \]
		
		\item Let $f \colon X \to Y$ be a proper morphism of derived algebraic Artin stacks, locally almost of finite presentation over $\mathbb C$.
		Then the diagram
		\[ \begin{tikzcd}
			\Coh^-(X) \arrow{r}{(-)\an} \arrow{d}{f_*} & \Coh^-(X\an) \arrow{d}{f\an_*} \\
			\Coh^-(Y) \arrow{r}{(-)\an} & \Coh^-(Y\an)
		\end{tikzcd} \]
		commutes.
		
		\item Let $X$ be a derived algebraic Artin stack, which is proper over $\mathbb C$.
		Then the analytification functor induces an equivalence between the unbounded stable $\infty$-categories
		\[ \Coh(X) \simeq \Coh^-(X\an) \]
		of coherent sheaves.
	\end{enumerate}
\end{thm-intro}

Finally, in \cref{sec:applications} we provide two applications of these results:

\begin{thm-intro}[cf.\ {\cref{prop:algebraizable} and \cref{cor:analytic_FMP}}]
	Let $X$ be a derived analytic space.
	Then:
	\begin{enumerate}
		\item if $X$ is proper over $\mathbb C$, then $X$ is algebraizable if and only if $\trunc(X)$ is algebraizable;
		\item let $x$ be a point of $X$. Then the shifted tangent complex $\mathbb T_xX[-1]$ admits a differential graded Lie algebra structure.
	\end{enumerate}
\end{thm-intro}

\bigskip
\paragraph{\textbf{Relation to other works}}
This is the first paper of the series \cite{Porta_Comparison_2017,Porta_Yu_DNAnG_I,Porta_Yu_Representability,Porta_RH}.
In revising it, I tried to use notations compatible with the conventions of those other papers.
Nevertheless, this paper remains essentially autonomous and self contained, with two exceptions: on one hand we use occasionally some results concerning $\infty$-categories and geometric stacks that appeared as general statements in \cite{Porta_Comparison_2017,Porta_Yu_DNAnG_I}.
On the other hand in the final section on applications of the GAGA theorem, we also rely on the properties of the analytic cotangent complex as introduced in the author's thesis (and that appeared in \cite{Porta_Yu_Representability}).

The results of Sections \ref{sec:functor_of_points} and \ref{sec:analytification} have non-archimedean counterparts that have been studied in \cite{Porta_Yu_DNAnG_I,Porta_Yu_Representability}.
The proofs we provide here for the complex analytic setting are different and more streamlined.

The debt to the work of J.\ Lurie should be evident to everyone.
As privately communicated, parts of these results exposed here will appear in the forthcoming book \cite{Lurie_SAG}.
They were achieved independently, and therefore many proofs differ.
Furthermore, the study of the case of Artin stacks (relying on \cite{Porta_Yu_Higher_analytic_stacks_2014}) is exclusive to this work.

Finally, a different approach to derived analytic geometry has been suggested by O.\ Ben-Bassat and K.\ Kremnizer \cite{Bambozzi_Dagger_2015,Bambozzi_BenBassat_Stein,Ben-Bassat_Non-archimedean_2013,Ben-Bassat_Perspective}. \\

\paragraph{\textbf{Notations and conventions}}

We use the language of $\infty$-categories.
We refer to \cite{HTT,Lurie_Higher_algebra} for the foundations and the terminology.
We reserve the notation $\cS$ to denote the $\infty$-category of spaces (that is, the underlying $\infty$-category of simplicial sets, endowed with the Kan model structure).
All functors are $\infty$-functors unless otherwise specified.
We refer to \cite[\S 2]{Porta_Yu_Higher_analytic_stacks_2014} for the terminology and the basic operations on $\infty$-Grothendieck sites.
Given an $\infty$-Grothendieck site $(\cC, \tau)$, we denote by $\Sh(\cC, \tau)$ the $\infty$-category of $\cS$-valued sheaves and by $\St(\cC, \tau)$ the full subcategory of $\Sh(\cC, \tau)$ spanned by hypercomplete sheaves.
We refer to a hypercomplete sheaf on $(\cC, \tau)$ as a \emph{stack}.

In this paper we are working in the setting of complex analytic geometry. Whenever the word ``analytic'' is used, it should be understood as ``complex analytic''.
We denote by $\bbA^n_{\mathbb C}$ the algebraic $n$-dimensional affine space and by $\bA^n_{\mathbb C}$ the analytic one.

\bigskip
\paragraph{\textbf{Acknowledgments}}

I am thankful to J.\ Lurie, V.\ Melani, M.\ Robalo, C.\ Simpson, B.\ To\"en, G.\ Vezzosi and T.\ Y.\ Yu for many stimulating discussions.
I would also like to thank V.\ F.\ Zenobi for introducing me to the idea of holomorphic functional calculus.
This research has been partially supported by the Simons Foundation grant number 347070.

\section{Review of derived complex analytic spaces} \label{sec:review}

We follow the ideas introduced in \cite{DAG-V,DAG-IX}.
First, we recall the notions of pregeometry and structured topos introduced by Lurie in \cite{DAG-V}.

\begin{defin}[{\cite[3.1.1]{DAG-V}}]
	A \emph{pregeometry} is an \infcat $\cT$ equipped with a class of \emph{admissible} morphisms and a Grothendieck topology generated by admissible morphisms, satisfying the following conditions:
	\begin{enumerate}[(i)]
		\item The \infcat $\cT$ admits finite products.
		\item The pullback of an admissible morphism along any morphism exists, and is again admissible.
		\item For morphisms $f,g$, if $g$ and $g\circ f$ are admissible, then $f$ is admissible.
		\item Every retract of an admissible morphism is admissible.
	\end{enumerate}
\end{defin}

\begin{defin}[{\cite[3.1.4]{DAG-V}}] \label{def:structure}
	Let $\cT$ be a pregeometry, and let $\cX$ be an \inftopos.
	A \emph{$\cT$-structure} on $\cX$ is a functor $\cO\colon\cT\to\cX$ with the following properties:
	\begin{enumerate}[(i)]
		\item The functor $\cO$ preserves finite products.
		\item Suppose given a pullback diagram
		\[
		\begin{tikzcd}
		U' \arrow{r} \arrow{d} & U \arrow{d}{f} \\
		X' \arrow{r} & X
		\end{tikzcd}
		\]
		in $\cT$, where $f$ is admissible.
		Then the induced diagram
		\[
		\begin{tikzcd}
		\cO(U') \arrow{r} \arrow{d} & \cO(U) \arrow{d} \\
		\cO(X') \arrow{r} & \cO(X)
		\end{tikzcd}
		\]
		is a pullback square in $\cX$.
		\item Let $\{U_\alpha\to X\}$ be a covering in $\cT$ consisting of admissible morphisms.
		Then the induced map
		\[\coprod_\alpha\cO(U_\alpha)\to\cO(X)\]
		is an effective epimorphism in $\cX$.
		\personal{This condition implies in particular that the germs of $\cO$ are local rings, local with respect to the topology generated by admissible morphisms.}
	\end{enumerate}
	We denote by $\Str_\cT(\cX)$ the \infcat of $\cT$-structures on $\cX$ and natural transformations between them.

	A morphism of $\cT$-structures $\cO\to\cO'$ on $\cX$ is \emph{local} if for every admissible morphism $U\to X$ in $\cT$, the resulting diagram
	\[ \begin{tikzcd}
	\cO(U) \arrow{r} \arrow{d} & \cO'(U) \arrow{d} \\
	\cO(X) \arrow{r} & \cO'(X)
	\end{tikzcd} \]
	is a pullback square in $\cX$.
	We denote by $\Strloc_\cT(\cX)$ the \infcat of $\cT$-structures on $\cX$ with local morphisms.
	
	A \emph{$\cT$-structured \inftopos} $X$ is a pair $(\cX,\cO_X)$ consisting of an \inftopos $\cX$ and a $\cT$-structure $\cO_X$ on $\cX$.
	We denote by $\RTop(\cT)$ the \infcat of $\cT$-structured \inftopoi (cf.\ \cite[Definition 1.4.8]{DAG-V}).
	Note that a 1-morphism $f\colon (\cX, \cO_X) \to (\cY, \cO_Y)$ in $\RTop(\cT)$ consists of a geometric morphism of \inftopoi $f_*\colon\cX\rightleftarrows\cY\colon f\inv$ and a local morphism of $\cT$-structures $f^\sharp \colon f\inv \cO_Y \to \cO_X$.
\end{defin}

Let $k$ denote either the field $\C$ of complex numbers or a non-archimedean field with nontrivial valuation.
We introduce three pregeometries $\cTan$, $\cTdisc$ and $\cTet$ that are relevant to derived complex analytic geometry.

The pregeometry $\cTan$ is defined as follows:
\begin{enumerate}[(i)]
	\item The underlying category of $\cTan$ is the category of open subsets of $\mathbb C^n$, for some $n \ge 0$;
	\item A morphism in $\cTan$ is admissible if and only if it is a local biholomorphism;
	\item The topology on $\cTan$ is the analytic topology.
\end{enumerate}
Given an $\infty$-topos $\cX$, we refer to $\cTan$-structures as \emph{analytic rings} and we write $\AnRing_{\mathbb C}(\cX)$ instead of $\Strloc_{\cTan}(\cX)$.

The pregeometry $\cTdisc$ is defined as follows:
\begin{enumerate}[(i)]
	\item The underlying category of $\cTdisc$ is the full subcategory of the category of $\mathbb C$-schemes spanned by affine spaces $\mathbb A^n_{\mathbb C}$;
	\item A morphism in $\cTdisc$ is admissible if and only if it is an isomorphism;
	\item The topology on $\cTdisc$ is the trivial topology, i.e.\ a collection of admissible morphisms is a covering if and only if it is nonempty.
\end{enumerate}
Given an $\infty$-topos $\cX$, we refer to $\cTdisc$-structures as \emph{derived rings} (or simplicial commutative rings) and we write $\CRing_{\mathbb C}(\cX)$ instead of $\Strloc_{\cTdisc}(\cX)$.

The pregeometry $\cTet$ is defined as follows:
\begin{enumerate}[(i)]
	\item The underlying category of $\cTet$ is the full subcategory of the category of complex schemes spanned by smooth affine schemes;
	\item A morphism in $\cTet$ is admissible if and only if it is étale;
	\item The topology on $\cTet$ is the étale topology.
\end{enumerate}

\begin{eg} \label{eg:Tdisc}
	In the case of $\cTdisc$ it is easy to give an interpretation of $\cTdisc$-structures.
	Indeed, the rectification result \cite[5.5.9.3]{HTT} allows to identify $\Str_{\cTdisc}(\cS)$ with the $\infty$-category of simplicial $\mathbb C$-algebras.
	See \cite[Remark 4.1.2]{DAG-V} for more details.
\end{eg}

\begin{eg}
	Let $X$ be an a complex analytic space.
	Let $\cX \coloneqq \Sh(X)$, the $\infty$-topos of sheaves of spaces over the underlying topological space of $X$.
	We define a $\cTan$-structure $\cO_X$ on $\cX$ as follows. Let $\mathrm{Op}(X)$ be the poset of open subsets of $X$.
	Then we define a functor
	\[ \cTan \times \mathrm{Op}(X)\op \to \cS \]
	by
	\[ (U,V) \mapsto \Hom_{\Anc}(V, U) . \]
	This determines a functor $\cTan \to \PSh(X)$ that is easily seen to factor through $\Sh(X)$ and to satisfy conditions (i), (ii) and (iii) of \cref{def:structure}.
\end{eg}

The classical analytification functor of \cite[Expos\'e XII]{SGA1} produces a transformation of pregeometries
\[ (-)\an \colon \cTdisc \to \cTan . \]
For any $\infty$-topos $\cX$, it induces a functor
\[ (-)\alg \colon \AnRing_{\mathbb C}(\cX) \to \CRing_{\mathbb C}(\cX) . \]
We refer to $(-)\alg$ as the \emph{underlying algebra functor}.

\begin{defin} \label{def:derived_analytic_space}
	A \emph{derived complex analytic space} $X$ is a $\cTan$-structured \inftopos $(\cX,\cO_X)$ such that there exists an effective epimorphism from $\coprod_i U_i$ to a final object of $\cX$ satisfying the following conditions, for every index $i$:
	\begin{enumerate}[(i)]
		\item The pair $(\cX_{/U_i}, \pi_0(\cO\alg_\cX | U_i))$ is equivalent to the ringed \inftopos associated to the étale site on a \kanal space $X_i$.
		\item For each $j\ge 0$, $\pi_j(\cO\alg_\cX | U_i)$ is a coherent sheaf of $\pi_0(\cO\alg_\cX | U_i)$-modules on $X_i$.
	\end{enumerate}
	We denote by $\dAnc$ the full subcategory of $\RTop(\cTan)$ spanned by derived complex analytic spaces.
\end{defin}

We can summarize the main results of \cite[\S 12]{DAG-IX} in the following theorem:

\begin{thm}[J.\ Lurie]
	The $\infty$-category $\dAnc$ admits fiber products and it contains as a full subcategory the ordinary category of complex analytic spaces $\Anc$.
	Furthermore, a derived complex analytic space $X = (\cX, \cO_X)$ belongs to $\Anc$ if and only if the $\infty$-topos $\cX$ is $0$-localic and the structure sheaf $\cO_X$ is discrete.
\end{thm}

\section{The functor of points} \label{sec:functor_of_points}

In \cref{def:derived_analytic_space} a derived analytic space is in particular a $\cTan$-structured topos.
This perspective is useful, but it has one main drawback: at best it allows us to deal with derived orbifolds.
On the other hand, it is often useful to consider more general geometric objects, such as derived Artin stacks.
For this reason, in this section we provide a functor of point approach to derived analytic spaces.
We follow the conventions of \cite[\S 2]{Porta_Yu_Higher_analytic_stacks_2014} for what concerns geometric contexts and geometric stacks. \\

Before introducing the geometric context, let us prove the following useful result:

\begin{lem} \label{lem:hypercomplete_underlying_topos}
	Let $X \coloneqq (\cX, \cO_X)$ be a derived complex analytic space.
	Then the $\infty$-topos $\cX$ is hypercomplete.
\end{lem}

\begin{proof}
	Using \cite[6.5.2.21 and 6.5.2.22]{HTT} we see that the condition of being hypercomplete can be tested locally on $\cX$.
	In particular, we can suppose that $\cX$ is $0$-localic and that $(\cX, \pi_0(\cO_X\alg))$ is an usual analytic space.
	In this case, $\cX$ is identified with the $\infty$-topos of sheaves on a locally compact Hausdorff space of finite covering dimension.
	In particular, its homotopy dimension is finite, and hence $\cX$ is hypercomplete (see \cite[7.2.1.12]{HTT}).
\end{proof}

\subsection{The geometric context} \label{subsec:geometric_context}

We start by introducing the notion of derived Stein space.

\begin{defin}
	Let $X \coloneqq (\cX, \cO_X)$ be a derived analytic space.
	We say that $X$ is a \emph{derived Stein space} if the truncation $\trunc(X) \coloneqq (\cX, \pi_0(\cO_X\alg))$ is an usual Stein space.
	We denote by $\dStn$ the full subcategory of $\dAnc$ spanned by derived Stein spaces.
\end{defin}

We next introduce the notion of \'etale and smooth morphism of derived analytic spaces.

\begin{defin} \label{def:differential_properties}
	Let $X \coloneqq (\cX, \cO_X)$ and $Y \coloneqq (\cY, \cO_Y)$ be derived analytic spaces.
	We say that a morphism $f \colon X \to Y$ is:
	\begin{enumerate}
		\item \emph{strong} if the induced morphism $f\inv \cO_Y\alg \to \cO_X\alg$ is strong, i.e.\ if it induces isomorphisms
		\[ \pi_i( f\inv \cO_Y\alg ) \simeq \pi_i( \cO_X\alg ) \otimes_{\pi_0 ( \cO_X\alg )} \pi_0( f\inv \cO_Y\alg ) , \]
		\item \emph{\'etale} if it is strong and, locally on $X$ and $Y$, the truncation $\trunc(f)$ is \'etale,
		\item \emph{smooth} if it is strong and the truncation $\trunc(f)$ is smooth.
	\end{enumerate}
\end{defin}

\begin{lem}
	A morphism $f \colon X \to Y$ of derived analytic spaces is \'etale if and only if it satisfies the following two conditions:
	\begin{enumerate}
		\item the underlying morphism of $\infty$-topoi $f_* \colon \cX \leftrightarrows \cY \colon f\inv$ is an \'etale morphism of $\infty$-topoi (see \cite[\S 6.3.5]{HTT}), and
		\item the induced morphism $f\inv \cO_Y \to \cO_X$ is an equivalence.
	\end{enumerate}
\end{lem}

\begin{proof}
	Suppose first that $f$ is \'etale. Then locally on $\cX$ and $\cY$, the geometric morphism $f_*$ is an \'etale morphism of $\infty$-topoi.
	It follows from the descent theory of $\infty$-topoi that $f_*$ is \'etale.
	Moreover, the map $f^\sharp \colon f\inv \cO_Y \to \cO_X$ induces an equivalence
	\[ f\inv ( \pi_0(\cO_Y\alg) ) \xrightarrow{\sim} \pi_0(\cO_X\alg) . \]
	Since the morphism is strong, we deduce that $f^\sharp$ induces isomorphisms
	\[ \pi_i(f\inv \cO_Y\alg) \simeq \pi_i(\cO_X\alg) . \]
	Since the underlying $\infty$-topos of $X$ is hypercomplete by \cref{lem:hypercomplete_underlying_topos} and since the underlying algebra functor $(-)\alg$ is conservative, we conclude that $f^\sharp$ was an equivalence to start with.
	
	Vice-versa, suppose that the above two conditions are met. Then in particular $f$ induces an equivalence
	\[ f\inv ( \pi_0\cO_Y ) \simeq \pi_0(\cO_X) . \]
	This implies that $\trunc(f)$ is locally on $X$ and $Y$ an open immersion. In other words, $f$ is \'etale.
\end{proof}

\begin{rem}
	It is possible to provide more conceptual characterization of smooth morphisms.
	Nevertheless, this requires the use of the \emph{analytic cotangent complex}.
	We refer to \cite[Proposition 5.48]{Porta_Yu_Representability} for a precise formulation.
	In this paper, we will not need the alternative characterizations of smooth morphisms provided in loc.\ cit.
\end{rem}

The notion of \'etale morphism of derived analytic spaces gives rise to a Grothendieck topology on $\dAnc$ that restricts to a Grothendieck topology on $\dStn$.\footnote{The only non trivial point is to show that the intersection of derived Stein spaces is again Stein. This can be checked on truncations, where it is a direct consequence of \cite[\S 1.4.4]{Grauert_Coherent_1984}.
	\personal{Actually, this reference only deals with finite intersections. See Proposition 6.2.7 in A.\ Kaneko, \emph{Introduction to the theory of hyperfunctions} for the general case}}
We write $\tauan$ to denote this topology and we refer to it as the \emph{analytic topology}.

\begin{lem} \label{lem:subcanonical}
	The topology $\tauan$ on $\dStn$ is hyper-subcanonical.
\end{lem}

\begin{proof}
	Let $X = (\cX, \cO_X) \in \dStn$ and let $U_\bullet \to X$ be a $\tauan$-hypercovering.
	We have to prove that there is an equivalence
	\[ | U_\bullet | \simeq X , \]
	where $|-|$ denotes the geometric realization.
	For $n \in \mathbf \Delta$, let us represent the derived analytic space $U_n$ as $U_n = (\cX_n, \cO_{U_n})$.
	Using the fact that the inclusion $\AnRing_{\mathbb C}(\cX) \hookrightarrow \Fun(\cTan, \cX)$ is conservative and commutes with limits, we are therefore reduced to prove the following two statements:
	\begin{enumerate}[(i)]
		\item one has $| \cX_\bullet | \simeq \cX$ in $\RTop$;
		\item one has
		\[ \cO_X \simeq \lim_{n \in \mathbf \Delta} i_{n*} \cO_{U_n} \]
		in the $\infty$-category $\Fun(\cTan, \cX)$.
	\end{enumerate}
	The first statement follows from the fact that $\cX$ is hypercomplete and from the descent theory of $\infty$-topoi (see \cite[6.1.3.9(3)]{HTT}).
	The second statement follows from the fact that limits in $\Fun(\cTan,\cX)$ can be computed objectwise and from descent theory of $\infty$-topoi.
\end{proof}

\begin{lem}
	The inclusion of sites $i \colon (\dStn, \tauan) \hookrightarrow (\dAnc, \tauan)$ induces an equivalence
	\[ \St( \dStn, \tauan ) \simeq \St( \dAnc, \tauan ) . \]
\end{lem}

\begin{proof}
	It is enough to prove that the conditions of \cite[Proposition 2.22]{Porta_Yu_Higher_analytic_stacks_2014} are met.
	For this, it is enough to prove that every derived analytic space admits a covering made of derived Stein spaces.
	Observe that there is an equivalence of $\infty$-categories between the $\infty$-category of \'etale maps to $X$ and the one of \'etale maps of $\trunc(X)$.
	Since the property of being Stein only depends on the truncation, we are reduced to prove the above statement for ordinary analytic spaces, where it is well known.
\end{proof}

We now consider the restricted Yoneda embedding, defined as the composition
\[ \begin{tikzcd} [column sep = small]
	\dAnc \arrow{r}{j} & \St(\dAnc, \tauan) \arrow{r}{i_s}[swap]{\sim} & \St( \dStn, \tauan ) .
\end{tikzcd} \]
We denote this composition by $\phi$.
The above results imply:

\begin{cor}
	The functor
	\[ \phi \colon \dAnc \to \St(\dStn, \tauan) \]
	is fully faithful.
\end{cor}

\subsection{Geometric stacks}

Our next goal is to describe the essential image of $\phi$.
We start by introducing some terminology.
We let $\bP\et$ denote the collection of \'etale morphisms in $\dStn$.
Then, $(\dStn, \tauan, \bP\et)$ is a geometric context in the sense of \cite[Definition 2.2]{Porta_Yu_Higher_analytic_stacks_2014}.
We refer to geometric stacks with respect to this context as \emph{derived analytic \DM stacks}.
We denote by $\mathrm{DM}$ the full subcategory of $\St(\dStn, \tauan)$ spanned by derived analytic \DM stacks.

\begin{defin}
	Let $F \in \St(\dStn, \tauan)$.
	We say that $F$ is \emph{$n$-truncated} if for every ordinary Stein space $X \in \dStn$, the space $F(X) \in \cS$ is $n$-truncated.
	We denote by $\mathrm{DM}_n$ the full subcategory of $\mathrm{DM}$ spanned by $n$-truncated derived analytic \DM stacks.
\end{defin}

On the other hand, we have:

\begin{notation}
	We let $\dAnc^{\le n}$ denote the full subcategory of $\dAnc$ spanned by those derived analytic spaces $X = (\cX, \cO_X)$ whose underlying $\infty$-topos $\cX$ is $n$-localic.
\end{notation}

The following theorem is the complex analytic version of \cite[Theorem 1.7]{Porta_Comparison_2017} (in the algebraic setting) and of \cite[Theorem 7.9]{Porta_Yu_DNAnG_I} (in the non-archimedean setting).

\begin{thm} \label{thm:functor_of_points}
	For every $n \ge 0$, the functor $\phi \colon \dAnc \to \St(\dStn, \tauan)$ restricts to an equivalence
	\[ \dAnc^{\le n} \simeq \mathrm{DM}_n . \]
\end{thm}

Since $\phi$ is fully faithful, it is enough to prove the following two statements:
\begin{enumerate}
	\item if $X = (\cX, \cO_X) \in \dAnc^{\le n}$, then $\phi(X)$ belongs to $\mathrm{DM}_n$, and
	\item the induced functor $\dAnc^{\le n} \to \mathrm{DM}_n$ is essentially surjective.
\end{enumerate}
We start by proving statement (1).

\begin{lem} \label{lem:functor_points_atlas}
	Let $n \ge 0$ and let $X = (\cX, \cO_X) \in \dAnc^{\le n}$.
	Then $\phi(X)$ is $n$-truncated and admits an $n$-atlas.
\end{lem}

\begin{proof}
	Let $Y \coloneqq (\cY, \cO_Y)$ be an ordinary Stein space.
	We have to prove that $\Map_{\dAnc}(Y, X)$ is $n$-truncated.
	Consider the canonical map
	\[ p \colon \Map_{\dAnc}(Y, X) \longrightarrow \Map_{\RTop}(\cY, \cX) . \]
	Recall from \cite[Lemma 2.2]{Porta_Comparison_2017} that the $\infty$-category of $n$-localic topoi is categorically equivalent to an $(n+1)$-category.
	In particular, $\Map_{\RTop}(\cY, \cX)$ is $n$-truncated.
	As a consequence, it is enough to prove that for any geometric morphism $f_* \colon \cY \leftrightarrows \cX \colon f\inv$ the fiber of $p$ at $f_*$ is $n$-truncated.
	We can identify this fiber with the space
	\[ \Map_{\AnRing_{\mathbb C}(\cX)}(f\inv \cO_X, \cO_Y) . \]
	Since $\cO_Y$ is discrete, we see that this space is $0$-truncated.
	In conclusion, $\Map_{\dAnc}(Y,X)$ is $n$-truncated.
	It follows that $\phi(X)$ is $n$-truncated.
	
	We now prove that $\phi(X)$ admits an $n$-atlas.
	We proceed by induction on $n$.
	When $n = 0$, we deal first with the case where the underlying topological space of $X$ is Hausdorff.
	In this case, we can find an \emph{open} covering $\{U_i\}$ of $X$ such that each $U_i$ is a derived Stein space.
	It is then enough to prove that the induced morphisms
	\[ \phi(U_i) \to \phi(X) \]
	are representable by derived Stein spaces.
	In other words, we have to prove that for any derived Stein space $Y$ and any map $\phi(Y) \to \phi(X)$ there exists a derived Stein space $V_i$ such that
	\[ \phi(V_i) \simeq \phi(U_i) \times_{\phi(X)} \phi(Y) . \]
	Since $\phi$ commutes with limits and it is fully faithful, we are reduced to prove that the fiber product $U_i \times_X Y$ is a derived Stein space.
	This statement can be checked on truncations.
	Since $\trunc(U_i)$ and $\trunc(Y)$ are Stein spaces, the conclusion follows from \cite[\S 1.4.4]{Grauert_Coherent_1984} (here we are using the Hausdorff assumption).
	
	Notice that above we proved that the morphisms $\phi(U_i) \to \phi(X)$ are $(-1)$-representable.
	If the underlying topological space of $X$ is not Hausdorff, we first find an open covering $\{U_i\}$ of $X$ made of derived Stein spaces whose underlying topological space is Hausdorff. Then for any derived Stein space $Y$ and any morphism $Y \to X$, the fiber product $Y \times_X U_i$ is an open subspace of $Y$ and it is therefore Hausdorff.
	It follows that in this case $\phi(U_i) \to \phi(X)$ is $0$-representable.
	
	Assume now $n > 0$.
	Choose an \'etale covering $\{U_i\}$ of $X$ such that each $U_i$ is a derived Stein space.
	Applying the induction hypothesis and reasoning as above, we are reduced to prove that for any derived Stein space $Y$ and any morphism $Y \to X$, the fiber product $U_i \times_Y X$ belongs to $\dAnc^{\le n-1}$.
	Using \cite[Lemma 2.3.16]{DAG-V} and \cite[6.3.5.8]{HTT}, we are reduced to prove that the underlying $\infty$-topos of $U_i$ can be written in the form $\cX_{/V_i}$, where $V_i \in \cX$ is an $(n-1)$-truncated object.
	Using the universal property of \'etale morphisms (see \cite[Remark 2.3.4]{DAG-V}), we reduce ourselves to prove that for any ordinary Stein space $Y$, the fibers of
	\[ \Map_{\dAnc}(Y, U_i) \to \Map_{\dAnc}(Y, X) \]
	are $(n-1)$-truncated.
	Since $\Map_{\dAnc}(Y,X)$ is $n$-truncated and $\Map_{\dAnc}(Y, U_i)$ is $0$-truncated, the conclusion follows from the long exact sequence of homotopy groups.
\end{proof}

To complete the proof of statement (1), we are therefore left to prove that the diagonal of $\phi(X)$ is $m$-representable for some $m$.
Rather than do it directly, it is easier to consider the geometric context $(\dAnc^{\le 0}, \tauan, \bP\et)$.
The inclusion
\[ i \colon (\dStn, \tauan) \hookrightarrow (\dAnc^{\le 0}, \tauan) \]
induces the equivalence
\[ \St( \dStn, \tauan ) \simeq \St( \dAnc^{\le 0}, \tauan ) . \]
The advantage is that $(\dAnc^{\le 0}, \tauan)$ is closed under $\tauan$-descent, in the sense of \cite[Definition 8.3]{Porta_Yu_DNAnG_I}.
In virtue of \cite[Corollary 8.6]{Porta_Yu_DNAnG_I} and of \cref{lem:functor_points_atlas}, we are therefore reduced to check that $i$ induces an equivalence between the $\infty$-categories of geometric stacks.
Using \cite[Corollary 2.26]{Porta_Yu_Higher_analytic_stacks_2014}, we are reduced to prove the following fact:

\begin{lem}
	Let $X \in \dAnc^{\le 0}$.
	Then $\phi(X)$ is a $1$-geometric stack with respect to the context $(\dStn, \tauan, \bP\et)$.
\end{lem}

\begin{proof}
	Let $X \in \dAnc^{\le 0}$ and suppose first that the underlying topological space of $X$ is Hausdorff.
	Then we already showed in the proof of \cref{lem:functor_points_atlas} that $\phi(X)$ admits a $(-1)$-atlas.
	Let us prove that the diagonal $\phi(X) \to \phi(X) \times \phi(X)$ is also $(-1)$-representable.
	Since $\phi$ commutes with limits, we are immediately reduced to prove that if $Y$ is a derived Stein space and $Y \to X \times X$ is any map, then the fiber product
	\[ Y \times_{X \times X} X \]
	is a derived Stein space.
	Since the underlying topological space of $X$ is Hausdorff, the diagonal $X \to X \times X$ is a closed immersion.
	In particular, $Y \times_{X \times X} X \to Y$ is also a closed immersion.
	It follows that the above fiber product is a derived Stein space, and therefore that $\phi(X)$ is $0$-geometric.
	
	Choose now an open cover $\{U_i\}$ of $X$ by derived Stein spaces.
	Then arguing as in the proof of \cref{lem:functor_points_atlas} we see that for any derived Stein space $Y$ and any map $Y \to X$ the fiber product $Y \times X U_i$ is Hausdorff. In particular, $\phi(Y \times_X U_i)$ is $0$-geometric.
	Furthermore, for any morphism $Y \to X \times X$, the fiber product $Y \times_{X \times X} X$ is locally closed inside $Y$ and it is therefore Hausdorff.
	It follows that the diagonal is $0$-geometric as well.
	In conclusion, $\phi(X)$ is $1$-geometric.
\end{proof}

This shows that $\phi \colon \dAnc \to \St(\dStn, \tauan)$ induces a fully faithful functor
\[ \phi \colon \dAnc^{\le n} \to \mathrm{DM}_n . \]
We are left to prove that $\phi$ is essentially surjective.

Let now $F \in \mathrm{DM}_n$.
Then the small \'etale site $F\et$ of $F$ is equivalent to the small \'etale site of $\trunc(F)$.
Since $F$ is $n$-truncated, it follows from \cite[Proposition 8.2]{Porta_Yu_DNAnG_I} that $\trunc(F)\et$ is categorically equivalent to an $n$-category.
In particular, the $\infty$-topos
\[ \cX_F \coloneqq \Sh( \trunc(F)\et, \tauan ) \]
is $n$-localic.

Consider now the composition
\[ \cTan \times (F\et)^{\mathrm{op}} \to \dAnc \times (\dAnc)^{\mathrm{op}} \xrightarrow{\Map} \cS . \]
This induces a well defined functor
\[ \cO_F \colon \cTan \to \PSh(F\et) \]
which is hypercomplete in virtue of \cref{lem:subcanonical}.

\begin{lem}
	Keeping the above notations, $\cO_F$ commutes with products, admissible pullbacks and takes $\tau$-covers to effective epimorphisms.
	In other words, $\cO_F$ defines a $\cTan$-structure on $\cX_F$.
\end{lem}

\begin{proof}
	The functor $\cTan \to \dAnc$ commutes with products and admissible pullbacks by construction. Moreover, it takes $\tau$-covers to effective epimorphisms in virtue of \cref{lem:subcanonical}.
\end{proof}

If $\{U_i \to F\}$ is an \'etale atlas of $F$, each $U_i$ defines an object $V_i$ in $\cX_F$. 
Unraveling the definitions, we see that the $\cTan$-structured topos $((\cX_F)_{/V_i}, \cO_F |_{V_i})$ is canonically isomorphic to $U_i \in \dAnc$ itself.
Therefore $X \coloneqq (\cX_F, \cO_F)$ belongs to $\dAnc^{\le n}$.

We are left to prove that $\phi(X') \simeq F$. We can proceed by induction on the geometric level $n$ of $X$.
If $n = -1$, the statement is clear.
If $n \ge 0$, the proof of \cref{lem:functor_points_atlas} shows that the \v{C}ech nerve of $\coprod U_i \to X$ is a groupoid presentation for $\phi(X')$. Since $\phi$ commutes with \v{C}ech nerves of \'etale maps and their realizations, we conclude that $\phi(X)$ is equivalent to $F$ itself.
The proof of \cref{thm:functor_of_points} is thus achieved.

\section{Analytification functor} \label{sec:analytification}

The transformation of pregeometries
\[ (-)\an \colon \cTet \to \cTan \]
induces a forgetful functor
\[ (-)\alg \colon \RTop(\cTan) \to \RTop(\cTet) . \]
In virtue of \cite[Theorem 2.1.12]{DAG-V}, this functor admits a left adjoint, denoted
\[ (-)\an \colon \RTop(\cTet) \to \RTop(\cTan) . \]
It follows from \cite[Proposition 2.3.18]{DAG-V} that $(-)\an$ takes $\cTet$-schemes to $\cTan$-schemes.
However, if we can identify on one side $\cTet$-schemes with complex derived \DM stacks, the same cannot be said for $\cTan$-schemes and derived analytic spaces.
Therefore, we have to prove that $(-)\an$ takes complex derived \DM stacks locally almost of finite presentation to derived analytic spaces.
This what we achieve in this section. Along the way, we establish the comparison with the classical analytification. \\

Let us start by recalling the notion of complex derived \DM stack almost of finite presentation:

\begin{defin} \label{def:almost_finite_presentation}
	Let $X \coloneqq (\cX, \cO_X)$ be a complex derived \DM stack.
	We say that $X$ is \emph{locally almost of finite presentation} if $(\cX, \pi_0( \cO_X))$ is a \DM stack of finite presentation and each $\pi_i(\cO_X)$ is coherent as sheaf of $\pi_0(\cO_X)$-modules.
\end{defin}

\begin{rem}
	The above definition differs from \cite[Definition 7.2.4.26]{Lurie_Higher_algebra}.
	However, the two definitions are equivalent.
	Indeed, the question is local on $X$ and we can therefore suppose that $X$ is affine.
	Let us write $X = \Spec(A)$.
	It follows from Theorem 7.4.3.18 in loc.\ cit.\ that $A$ is almost of finite presentation in the sense of \cite{Lurie_Higher_algebra} if and only if $\pi_0(A)$ is of finite presentation over $\mathbb C$ and the cotangent complex $\mathbb L_{A / \mathbb C}$ is almost perfect.
	Using the equivalence
	\[ \pi_{n+2}( \mathbb L_{\tau_{\le n} A / \tau_{\le n +1} A} ) \simeq \pi_{n+1}(A) \]
	provided by \cite[Lemma 2.2.2.8]{HAG-II} and proceeding by induction on the Postnikov tower, it is easy to see that if $A$ is almost of finite presentation then $A$ in the sense of \cite{Lurie_Higher_algebra}, then it is almost of finite presentation in the sense of the above definition.
	The converse follows from the fact that $\pi_0(A)$ is noetherian (being of finite presentation over $\mathbb C$) and from the derived Hilbert's basis theorem, \cite[Theorem 7.2.4.31]{Lurie_Higher_algebra}.
\end{rem}

\begin{lem}
	Let $X \coloneqq (\cX, \cO_X)$ be a complex derived \DM stack almost of finite presentation.
	Set $X\an \coloneqq (\cX\an, \cO_{X\an}) \in \RTop(\cTan)$.
	Then if $\cO_X$ is $n$-truncated, the same goes for $\cO_{X\an}$.
\end{lem}

\begin{proof}
	This follows directly from \cref{prop:relative_spectrum_truncated_objects}.
\end{proof}

\begin{lem} \label{lem:analytification_discrete}
	Let $X \coloneqq (\cX, \cO_X)$ be a complex derived \DM stack almost of finite presentation.
	If $\cO_X$ is discrete, then $X\an$ is a derived analytic space, and it coincides with the classical analytification of $X$.
\end{lem}

\begin{proof}
	Using \cite[2.1.3]{DAG-V} we can reason locally on $X$ and therefore assume that $X$ is affine.
	In this case, we can find an underived pullback diagram of the form
	\[ \begin{tikzcd}
		X \arrow{r} \arrow{d} & \bbA^n_{\mathbb C} \arrow{d} \\
		\Spec(\mathbb C) \arrow{r}{0} & \bbA^m_{\mathbb C} .
	\end{tikzcd} \]
	Let $X'$ be the derived pullback of the above diagram.
	Then $\trunc(X') \simeq X$.
	Moreover, since the bottom morphism in the above square is a closed immersion, it follows from \cite[Proposition 4.1]{DAG-IX} that $X'$ is the pullback of the above diagram in $\RTop(\cTet)$.
	In particular, we have a pullback diagram
	\[ \begin{tikzcd}
		(X')\an \arrow{r} \arrow{d} & (\mathbb A^n_{\mathbb C})\an \arrow{d} \\
		\Spec(\mathbb C)\an \arrow{r} & (\bbA^m_{\mathbb C})\an .
	\end{tikzcd} \]
	Since $\bbA^r_{\mathbb C} \in \cTet$ for every $r \ge 0$, and since the classical analytification of $\bbA^r_{\mathbb C}$ is $\bA^r_{\mathbb C}$, we conclude that
	\[ (\Spec(\mathbb C))\an \simeq *_\mathbb{C}, \quad (\bbA^n_{\mathbb C})\an \simeq \bA^n_{\mathbb C} , \quad (\bbA^m_{\mathbb C})\an \simeq \bA^m_{\mathbb C} . \]
	It follows from \cref{prop:relative_spectrum_and_truncations} that
	\[ \trunc((X')\an) \simeq \trunc(X')\an \simeq X\an . \]
	Finally, since $*_\mathbb{C}$, $\bA^n_{\mathbb C}$ and $\bA^m_{\mathbb C}$ are derived analytic spaces and since derived analytic spaces are closed under pullbacks along closed immersions in $\RTop(\cTan)$ (by \cite[Proposition 12.10]{DAG-IX}), we conclude that $X\an$ is a derived analytic space.
	The universal property of the analytification shows that it coincides with the classical analytification of $X$.
\end{proof}

\begin{thm} \label{thm:derived_analytification}
	Let $X$ be a complex derived \DM stack locally almost of finite presentation.
	Then $X\an$ is a derived analytic space.
\end{thm}

\begin{proof}
	Let $\cC$ be the full subcategory of derived \DM stacks locally almost of finite presentation spanned by those $X$ such that $X\an$ is a derived analytic space.
	Then:
	\begin{enumerate}
		\item in virtue of \cref{lem:analytification_discrete}, $\cC$ contains discrete \DM stacks locally almost of finite presentation;
		\item if $X$ admits an \'etale cover $\{U_i\}$ such that each $U_i$ belongs to $\cC$, then $X$ belongs to $\cC$. This follows from \cite[Lemma 2.1.3]{DAG-V};
		\item if $X = (\cX, \cO_X)$ is such that every truncation $\mathrm t_{\le n} X = (\cX, \tau_{\le n} \cO_X)$ belongs to $\cC$, then $X$ belongs to $\cC$.
		This is a direct consequence of Propositions \ref{prop:relative_spectrum_and_truncations} and \ref{prop:relative_spectrum_truncated_objects};
		\item $\cC$ is closed under pullbacks along closed immersions. This is because derived analytic spaces are closed under pullback along closed immersions in $\RTop(\cTan)$;
		\item $\cC$ is closed under finite limits. Indeed, since pullbacks can be constructed out of pullbacks along closed immersions and products of affine spaces, this fact follows from the previous point and the fact that the analytification of $\bbA^n_{\mathbb C}$ is $\bA^n_{\mathbb C}$.
	\end{enumerate}
	It follows that it is enough to prove that if $X$ is a derived affine space which is $n$-truncated and almost of finite presentation, then $X\an$ is a derived analytic space.
	In this case, we can write $X$ as retract of $\mathrm t_{\le n}(Y)$, where $Y$ is of finite presentation.
	As $Y$ can be written as a finite limit of a diagram taking values in the affine spaces, we conclude that $Y \in \cC$.
	Moreover, since we already know that $\trunc(X) \in \cC$, we are left to check that the sheaves $\pi_i(\cO_{X\an}\alg)$ are coherent.
	This follows because $X\an$ is retract of $(\mathrm t_{\le n} Y)\an$.
\end{proof}

\section{Local theory} \label{sec:local_theory}

This section is devoted to the study of the category of analytic rings.

\begin{defin} \label{def:analytic_ring}
	The \emph{$\infty$-category of analytic rings}, denoted $\AnRing_{\mathbb C}$, is by definition the full subcategory of $\Strloc_{\cTan}(\cS)$ spanned by those $\cTan$-structures $A$ such that $A(\mathbb C)$ is \emph{not} contractible.
\end{defin}

The transformation of pregeometries
\[ (-)\an \colon \cTet \to \cTan \]
induced by the classical analytification functor of \cite[Expos\'e XII]{SGA1} produces a forgetful functor
\[ (-)\alg \colon \AnRing_{\mathbb C} \to \CRing_{\mathbb C} , \]
whose essential image is contained in the full subcategory of $\CRing_{\mathbb C}$ spanned by those simplicial commutative rings $A$ such that $\pi_0(A)$ is local and strictly Henselian.
We refer to $(-)\alg$ as the \emph{underlying algebra functor}.
It follows from \cite[Proposition 11.9]{DAG-IX} that $(-)\alg$ is conservative.

\begin{rem}
	According to the above definition, the underlying algebra of an analytic ring is always a \emph{local} ring.
	It is possible to develop a \emph{global} theory of analytic rings (see \cite{Lurie_SAG}).
	As a consequence, it would be more precise to refer to the objects of $\AnRing_{\mathbb C}$ as \emph{local analytic rings}.
	However, since in this paper we are only concerned with local analytic rings, we decided to refer to them simply as analytic rings.
\end{rem}

\begin{rem}
	Let $A \in \Strloc_{\cTan}(\cS)$.
	Then $A(\mathbb C)$ is not contractible if and only if $A\alg$ is not the zero ring.
	Notice that if this implies that $\pi_0(A(\mathbb C))$ has at least two elements.
\end{rem}

As remarked at the beginning of \cref{subsec:local_structures}, $\AnRing_{\mathbb C}$ is not presentable, at least a priori.
The first goal of this section is to prove that, in this special case, this is true.

\subsection{Residue fields}

The field of complex numbers $\mathbb C$ can be naturally seen as an analytic ring.
Indeed, if we denote by $*_{\mathbb C}$ the complex analytic space reduced to a single point with $\mathbb C$ as structure sheaf, then
\[ \Hom_{\Anc}(*_{\mathbb C}, -) \colon \cTan \to \cS \]
is a discrete analytic ring, whose underlying algebra coincides with $\mathbb C$.
We will commit an abuse of notation and write $\mathbb C$ to denote both the analytic ring $\Hom_{\Anc}(*_{\mathbb C}, -)$ and the field of complex numbers.
The following theorem is an analogue in the setting of analytic rings of the well known fact that if a complex commutative Banach algebra is a field, then it is isomorphic to $\mathbb C$:

\begin{thm} \label{thm:residue_field}
	The forgetful functor
	\[ ( \AnRing_{\mathbb C} )_{/\mathbb C} \to \AnRing_{\mathbb C} \]
	is an equivalence of $\infty$-categories.
	In particular, if $A$ is an analytic ring, then the residue field of $A\alg$ is isomorphic to $\mathbb C$.
\end{thm}

Before giving the proof of the above theorem, let us describe two consequences:

\begin{cor} \label{cor:analytic_rings_presentable}
	The $\infty$-category of analytic rings is presentable.
\end{cor}

\begin{proof}
	It follows from \cref{cor:augmented_local_strutures_presentable} that $(\AnRing_{\mathbb C})_{/\mathbb C}$ is a presentable $\infty$-category.
	The statement is therefore a direct consequence of \cref{thm:residue_field}.
\end{proof}

Let $\CRing_{\mathbb C}^{\mathrm{aug}}$ denote the overcategory $(\CRing_{\mathbb C})_{/ \mathbb C}$.
Then \cref{thm:residue_field} implies that the underlying algebra functor factors as
\[ (-)\alg \colon \AnRing_{\mathbb C} \to \CRing_{\mathbb C}^{\mathrm{aug}} . \]
We have:

\begin{cor} \label{cor:local_analytification}
	The underlying algebra functor $(-)\alg \colon \AnRing_{\mathbb C} \to \CRing_{\mathbb C}^{\mathrm{aug}}$ admits a left adjoint.
\end{cor}

\begin{proof}
	This is a direct application of \cref{cor:augmented_local_structures_change_of_pregeometry}.
\end{proof}

We denote the left adjoint of $(-)\alg$ by
\[ (-)\an \colon \CRing_{\mathbb C}^{\mathrm{aug}} \to \AnRing_{\mathbb C} , \]
and refer to $(-)\an$ as the \emph{analytification functor}.
This terminology will be justified later. \\

We now start the proof of \cref{thm:residue_field}.
It is enough to prove that $\mathbb C$ is a final object of $\AnRing_{\mathbb C}$.
We start by recalling that, since $\cTan$ is compatible with $n$-truncations for $n \ge 0$,\footnote{See \cref{subsec:truncation_structures} for a review of this notion, and \cite[Proposition 11.4]{DAG-IX} for a proof of this statement.} if $A \in \AnRing_{\mathbb C}$, then $\pi_0(A)$ is again an analytic ring.
Moreover, $\mathbb C$ is a discrete analytic ring.
As a consequence, we obtain an canonical equivalence
\[ \Map_{\AnRing_{\mathbb C}}(A, \mathbb C) \simeq \Map_{\AnRing_{\mathbb C}}(\pi_0(A), \mathbb C) . \]
In particular, it is enough to prove that $\mathbb C$ is a final object in the category of discrete analytic rings.

\begin{notation}
	We denote by $\AnRing_{\mathbb C}^0$ the full subcategory of $\AnRing_{\mathbb C}$ spanned by discrete analytic rings.
\end{notation}

\begin{rem}
	We can identify $\AnRing_{\mathbb C}^0$ with the $1$-category of functors $A \colon \cTan \to \rSet$ that commutes with products and admissible pullbacks, takes coverings in $\cTan$ to effective epimorphisms and such that $A(\mathbb C)$ consists of at least two elements.
	In particular, a discrete analytic ring has a well defined underlying \emph{set} $A(\mathbb C)$.
	In what follows, we write $f \in A$ to denote that $f$ is an element in $A(\mathbb C)$.
\end{rem}

Taking inspiration from functional analysis, we next introduce the spectrum of an element $f \in A$ in an analytic ring.
Recall that if $B$ is an ordinary ring, $B^\times$ denotes the multiplicative subgroup spanned by invertible elements.

\begin{defin}
	Let $A \in \AnRing_{\mathbb C}$ and let $f\in A$.
	The \emph{spectrum of $f$} is the subset of $\mathbb C$ defined by
	\[ \sigma_A(f) \coloneqq \left\{ z \in \mathbb C \mid f - z \notin (A\alg)^\times \right\} \]
\end{defin}

We summarize the main properties of the spectrum of an element in the following lemma:

\begin{lem} \label{lem:sorite_functional_spectrum}
	Let $A \in \AnRing_{\mathbb C}$ and let $f \in A$.
	\begin{enumerate}
		\item $A(\emptyset) = \emptyset$;
		\item Let $z \in \mathbb C$. Then $z \notin \sigma_A(f)$ if and only if $f \in A(\mathbb C^*_z)$, where $\mathbb C^*_z \coloneqq \mathbb C \smallsetminus \{z\}$.
		\item The spectrum $\sigma_A(f)$ contains exactly one element.
	\end{enumerate}
\end{lem}

\begin{proof}
	We start by proving statement (1).
	Observe that $\emptyset \to *$ is an open immersion.
	It follows that $A(\emptyset) \subseteq A(*)$ and therefore that $A(\emptyset)$ has at most of one element.
	Suppose by contradiction that $A(\emptyset) \ne \emptyset$.
	Let $z \in \mathbb C$ and denote by $t_z \colon \mathbb C \to \mathbb C$ the translation by $z$.
	Then we have the following commutative triangle
	\[ \begin{tikzcd}
		{} & \emptyset \arrow{dl} \arrow{dr} \\
		\mathbb C \arrow{rr}{t_a} & & \mathbb C .
	\end{tikzcd} \]
	Applying $A$, we see that $A(\emptyset) \to \mathbb A(\mathbb C)$ selects an element $f \in A(\mathbb C)$ such that $f + z = f$ for every $z \in \mathbb C$.
	Since $A\alg$ is not the zero ring by assumption, this is impossible unless $z = 0$.
	It follows that $A(\emptyset) = \emptyset$.
	
	We now turn to statement (2).
	When $z = 0$, we simply write $\mathbb C^*$ instead of $\mathbb C^*_0$.
	It follows from \cite[Lemma 4.2.5]{DAG-V} that $A(\mathbb C^*) = (A\alg)^\times$.
	Suppose therefore that $f \in A(\mathbb C^*_z)$.
	Then $f - z \in A(\mathbb C^*) \simeq (A\alg)^\times$ and therefore $z \notin \sigma_A(f)$.
	Vice-versa, suppose that $z \notin \sigma_A(f)$.
	Then $f -  z \in (A\alg)^\times = A(\mathbb C^*)$ and as a consequence, $f \in A(\mathbb C^*_z)$.
	
	Finally, we prove statement (3). We start by observing that $\sigma_A(f)$ consists at most of one element.
	Indeed, if $z, w \in \sigma_A(f)$ are distinct, we can choose disjoint open neighborhoods $U$ and $V$ of $z$ and $w$ in $\mathbb C$, respectively.
	Thus, we have coverings
	\[ \mathbb C = U \cup \mathbb C^*_z , \qquad \mathbb C = V \cup \mathbb C^*_w . \]
	Since $A$ takes coverings to epimorphism, we conclude that $f$ belongs either to $A(U)$ or to $A(\mathbb C^*_z)$.
	Since point (2) shows that $f \notin A(\mathbb C^*_z)$, it follows that $f \in A(U)$.
	Similarly, the second covering we considered implies that $f \in A(V)$.
	In particular, 
	\[ f \in A(U) \cap A(V) \simeq A(U \cap V) = A(\emptyset) = \emptyset . \]
	This is a contradiction, and therefore we conclude that $\sigma_A(f)$ has at most one element.
	
	We are left to prove that $\sigma_A(f)$ is not empty. Suppose again by contradiction that $\sigma_A(f) = \emptyset$.
	Then point (2) implies that $f \in A(\mathbb C^*_z)$ for every $z \in \mathbb C$.
	Let $\varepsilon > 0$ be a real number and denote by $\overline{\mathrm D}(z,\varepsilon)$ (resp.\ $\mathrm D(z, \varepsilon)$) the closed disk centered at $z$ and of radius $\varepsilon$.
	Then
	\[ \mathbb C^*_z = \bigcup_{\varepsilon > 0} \mathbb C \smallsetminus \overline{\mathrm D}(z, \varepsilon) . \]
	Since $A$ takes coverings to effective epimorphisms, we conclude that there exists $\varepsilon_z > 0$ such that
	\[ f \in A\left( \mathbb C \smallsetminus \overline{\mathrm D}(z, \varepsilon_z) \right) . \]
	Since $\mathbb C \smallsetminus \overline{\mathrm D}(z, \varepsilon_z)$ is disjoint from $\mathrm D(z, \varepsilon_z)$, point (1) implies that
	\[ f \notin A \left( \mathrm D(z, \varepsilon_z) \right) . \]
	However, we have
	\[ \mathbb C = \bigcup_{z \in \mathbb C} \mathrm D(z, \varepsilon_z) . \]
	Since $A$ takes coverings to epimorphisms, we see that this produces a contradiction.
	Therefore $\sigma_A(f) \ne \emptyset$ and the proof is thus complete.
\end{proof}

\begin{cor} \label{cor:formal_properties_functional_spectrum}
	Let $A \in \AnRing_{\mathbb C}$.
	Then:
	\begin{enumerate}
		\item let $U \subset \mathbb C$ be an open. If $a \in A(U)$, then $\sigma_A(a) \in U$.
		
		\item Let $U \subseteq \mathbb C^n$ and let $(a_1, \ldots, a_n) \in A(U)$.
		Let $f \colon U \to \mathbb C^m$ be a holomorphic function and let $f_A \colon A(U) \to A(\mathbb C^m)$ be the induced function.
		Set $f_A(a_1, \ldots, a_n) \coloneqq (b_1, \ldots, b_n)$.
		Then
		\[ (\sigma_A(b_1), \ldots, \sigma_A(b_m)) = f( \sigma_A(a_1), \ldots, \sigma_A(a_n) ) . \]
	\end{enumerate}
\end{cor}

\begin{proof}
	We prove point (1). Set $z \coloneqq \sigma_A(a)$. Then \cref{lem:sorite_functional_spectrum}(2) implies that $a \notin A( \mathbb C^*_z )$.
	Observe that if $z \notin U$ then $A(U) \subset A(\mathbb C^*_z)$.
	Since $a \in A(U)$, we must therefore conclude that $z \in U$.
	
	We now turn to point (2).
	We can assume without loss of generality that $m = 1$.
	Furthermore, we can assume that $\sigma_A(a_i) = 0$ for $1 \le i \le n$.
	Point (1) guarantees that $0_n \coloneqq (0, \ldots, 0) \in U$.
	Set $w \coloneqq f(0_n)$.
	Up to replacing $f$ with $f - w$, we can also assume $w = 0$.
	\Cref{lem:sorite_functional_spectrum}(2) implies that $(a_1, \ldots, a_n) \in A( U \smallsetminus \{0_n\} )$.
	Since $A$ takes coverings to effective epimorphisms, we see that we can choose an open Stein neighborhood $U'$ of $(0_n)$ in $U$ such that $(a_1, \ldots, a_n) \in A(U')$.
	In other words, we can assume from the beginning that $U$ is Stein.
	
	Using \cref{lem:sorite_functional_spectrum}(3), it is enough to prove that
	\[ 0 \in \sigma_A(f_A(a_1, \ldots, a_n)) . \]
	Set
	\[ b \coloneqq f_A( a_1, \ldots, a_n ) . \]
	Then we have to show that $b$ is not invertible in $A$.
	Suppose by contradiction that it is.
	Since $U$ is Stein and $f(0_n) = 0$, we can find holomorphic functions $g_1, \ldots, g_n \colon U \to \mathbb C$ such that
	\[ f = g_1 z_1 + \ldots, g_n z_n , \]
	where for $1 \le i \le n$ we denote $z_i$ the $i$-th coordinate function on $\mathbb C^n$.
	In particular, we have
	\[ b = g_1(a_1, \ldots, a_n) a_1 + \ldots + g_n(a_1, \ldots, a_n) a_n . \]
	If $b$ is invertible, then the ideal generated by $a_1, \ldots, a_n$ is the unit ideal.
	Nevertheless, $A\alg$ is a local ring and $a_1, \ldots, a_n$ belong to the maximal ideal of $A\alg$.
	This is a contradiction.
\end{proof}

We are now ready to prove the main result of this section, which implies \cref{thm:residue_field} thanks to \cite[1.2.12.4]{HTT}:

\begin{prop}
	The analytic ring $\mathbb C$ is a final object for $\AnRing_{\mathbb C}^0$.
\end{prop}

\begin{proof}
	Let $A \in \AnRing_{\mathbb C}$.
	Let us temporarily denote by $H$ the analytic ring $\mathbb C$.
	We define a morphism of analytic rings $s_A \colon A \to H$ as follows.
	For $(f_1, \ldots, f_n) \in A(\mathbb C^n)$, we set
	\[ s_A(f_1, \ldots, f_n) \coloneqq ( \sigma_A(f_1), \ldots, \sigma_A(f_n) ) \in \mathbb C^n \simeq H(\mathbb C^n) . \]
	If $U \subset \mathbb C^n$ is an open subset and $(f_1, \ldots, f_n) \in A(U)$, \cref{cor:formal_properties_functional_spectrum}(1) guarantees that
	\[ s_A(f_1, \ldots, f_n) \in U . \]
	Furthermore, \cref{cor:formal_properties_functional_spectrum}(2) ensures that for any pair of open subsets $U \subset \mathbb C^n$ and $V \subset \mathbb C^m$ and any holomorphic map
	\[ f \colon U \to V , \]
	the diagram
	\[ \begin{tikzcd}
		A(U) \arrow{r}{f_A} \arrow{d}{s_A} & A(V) \arrow{d}{s_A} \\
		U \arrow{r}{f} & V
	\end{tikzcd} \]
	commutes.
	In other words, $s_A \colon A \to H$ is a natural transformation.
	
	Let us prove that $s_A$ is a local morphism of analytic rings.
	Unraveling the defiinitions, we see that it is enough to prove that if $U \subset \mathbb C$ is an open subset, then the diagram
	\[ \begin{tikzcd}
		A(U) \arrow[hook]{r} \arrow{d}{s_A} & A(\mathbb C) \arrow{d}{s_A} \\
		U \arrow[hook]{r} & \mathbb C
	\end{tikzcd} \]
	is a pullback square.
	Let therefore $a \in A(\mathbb C)$ and suppose that $z \coloneqq s_A(a) \in U$.
	Then \cref{lem:sorite_functional_spectrum}(2) guarantees that $f \notin A(\mathbb C^*_z)$.
	Since $z \in U$, we see that
	\[ \mathbb C = U \cup \mathbb C^*_z . \]
	Since $A$ takes coverings to epimorphism, we deduce that $a \in A(U)$.
	This proves that $s_A \colon A \to H$ is a local morphism of analytic rings.
	
	We are left to prove that $\sigma$ is unique.
	Equivalently, we prove that for any local morphism of analytic rings $f \colon A \to B$ one has $s_A = s_B \circ f$.
	Let $a \in A$ and set $z \coloneqq s_A(a)$.
	Then $a - z$ is not invertible.
	Using \cref{lem:sorite_functional_spectrum}(2) and the fact that $A$ takes coverings to effective epimorphisms, we deduce that $a-z$ belongs to $A( \mathrm D(0, \varepsilon))$ for every $\varepsilon > 0$.
	As a consequence, $f(a - z) = f(a) - z$ belongs to $B(\mathrm D(0, \varepsilon))$ for every $\varepsilon > 0$.
	Using \cref{cor:formal_properties_functional_spectrum}, we deduce that $s_B(f(a)) \in \mathrm D(z, \varepsilon)$ for every $\varepsilon > 0$.
	In other words, $s_B(f(a)) = z = s_A(a)$.
	This proves that $s_A = s_B \circ f$. The proof is therefore complete.
\end{proof}

\subsection{Analytification and flatness} \label{subsec:analytification_flatness}

We now turn our attention to the analytification functor
\[ (-)\an \colon \CRing_{\mathbb C}^{\mathrm{aug}} \to \AnRing_{\mathbb C} \]
introduced in \cref{cor:local_analytification}.
Our goal is to prove the following theorem.

\begin{thm} \label{thm:local_analytification_flat}
	Let $A \in \CRing_{\mathbb C}^{\mathrm{aug}}$.
	Then the canonical map
	\[ \eta_A \colon A \longrightarrow (A\an)\alg \]
	is flat in the derived sense.
	In other words $\pi_0(\eta_A)$ is flat and, for every $n \ge 0$, $\eta_A$ induces isomorphisms
	\[ \pi_n(A) \otimes_{\pi_0(A)} \pi_0((A\an)\alg) \simeq \pi_n((A\an)\alg) . \]
\end{thm}

Before giving the proof of this theorem, we need a more geometric characterization of the analytification functor introduced in \cref{cor:local_analytification}.
In order to obtain such description, it is convenient to \emph{temporarily} work with the complex Zariski pregeometry $\cT_{\mathrm{Zar}} = \cT_{\mathrm{Zar}}(\mathbb C)$ (see \cite[4.2]{DAG-V}).
Notice that $\bbA^n_{\mathbb C} \in \cT_{\mathrm{Zar}}$.
We can therefore consider it as a $\cT_{\mathrm{Zar}}$-structured topos.
Notice also that there is a transformation of pregeometries
\[ (-)\an \colon \cT_{\mathrm{Zar}} \longrightarrow \cTan \]
that factors through $\cTet$. In particular, the analytification of a derived scheme $X$ locally almost of finite presentation, considered as a $\cT_{\mathrm{Zar}}$-structured topos, coincides with the one obtained by considering $X$ as a $\cTet$-structured topos.

Let therefore $X$ be a complex derived \emph{scheme} locally almost of finite presentation, considered as a $\cT_{\mathrm{Zar}}$-structured topos.
Let $X\an \coloneqq (\cX\an, \cO_{X\an})$ be its analytification.
Let $x \colon \Spec(\mathbb C) \to X$ be a \emph{closed} point and let $z \colon *_{\mathbb C} \to X\an$ be the corresponding point of the analytification.
Let
\[ x_* \colon \cS \leftrightarrows \cX \colon x\inv , \qquad z_* \colon \cS \leftrightarrows \cX\an \colon z\inv \]
be the induced geometric morphisms of $\infty$-topoi.
We can consider $x\inv \cO_X$ as an object in $\CRing_{\mathbb C}^{\mathrm{aug}}$.
In particular, we can apply to $x\inv \cO_X$ the analytification functor introduced in \cref{cor:local_analytification}.
Notice that the canonical application
\[ p \colon (\cX\an, \cO_{X\an}) \to (\cX, \cO_X) \]
induces a morphism
\[ \alpha \colon x\inv \cO_X \longrightarrow z\inv \cO_{X\an} . \]
We have:

\begin{lem} \label{lem:local_analytification_affine_space}
	The morphism $\alpha \colon x\inv \cO_X \to z\inv \cO_{X\an}$ exhibits $z\inv \cO_{X\an}$ as analytification of $x\inv \cO_X$.
\end{lem}

\begin{proof}	
	Observe that both $\cX$ and $\cX\an$ are $0$-localic $\infty$-topoi.
	We abuse notation and denote by $X$ and $X\an$ their underlying topological spaces, and by $p \colon X\an \to X$ the continuous morphism between them.
	Notice that since $x \in X$ is a closed point, the geometric morphism $x_* \colon \cS \to \cX$ is a closed immersion of $\infty$-topoi.
	Furthermore, $p\inv(\{x\}) = \{z\}$.
	We can therefore deduce from \cite[7.3.2.13]{HTT} that the square
	\[ \begin{tikzcd}
		\cS \arrow{d}{z_*} \arrow[equal]{r} & \cS \arrow{d}{z_*} \\
		\Sh(X\an) \arrow{r} & \Sh(X)
	\end{tikzcd} \]
	is a pullback in $\RTop$ (here $z_*$ denotes the geometric morphism determined by the point $z$).
	In particular, we deduce from \cite[Lemma 2.1.3]{DAG-V} that $(\cS, z\inv \cO_{X\an})$ is the analytification of $(\cS, z\inv \cO_X)$.
	At this point, the universal property of the analytification allows to conclude.
\end{proof}

\begin{proof}[Proof of \cref{thm:local_analytification_flat}]
	Let $S^n \in \cS$ denote the $n$-sphere.
	The forgetful functor $\CRing_{\mathbb C} \to \cS$ has a left adjoint, denoted $\mathbb C[-]$.
	Every element in $A \in \CRing_{\mathbb C}$ can be written as the colimit of a diagram
	\[ A_0 \xrightarrow{f_0} A_1 \xrightarrow{f_1} \cdots \to A_\alpha \xrightarrow{f_\alpha} \cdots , \]
	where $A_0$ is a discrete augmented $\mathbb C$-algebra (not necessarily of finite presentation over $\mathbb C$) and and each morphism $f_\alpha$ fits in a pushout diagram
	\[ \begin{tikzcd}
		\mathbb C[S^n] \arrow{r} \arrow{d} & \mathbb C[X] \arrow{d} \\
		A_\alpha \arrow{r}{f_\alpha} & A_{\alpha + 1}
	\end{tikzcd} \]
	for some $n \ge 0$.
	Here $\mathbb C[X] \simeq \mathbb C[\{*\}] \simeq \mathbb C[\Delta^n]$ denotes the polynomial algebra on one generator in degree $0$.
	
	Since flat morphisms are stable under filtered colimits, and since the functor $(-)\alg$ commutes with filtered colimits by \cref{cor:augmented_local_structures_change_of_pregeometry}, we are reduced to prove the following statements:
	\begin{enumerate}
		\item the map $\eta_A$ is flat when $A$ is a discrete augmented $\mathbb C$-algebra;
		\item suppose given an integer $n \ge 0$ and a pushout diagram in $\CRing_{\mathbb C}^{\mathrm{aug}}$
		\[ \begin{tikzcd}
			\mathbb C[S^n] \arrow{r}{f} \arrow{d} & \mathbb C[X] \arrow{d} \\
			A \arrow{r} & B .
		\end{tikzcd} \]
		Then if the map $\eta_A$ is flat, the same goes for $\eta_B$.
	\end{enumerate}
	
	We start by proving (1).
	We can write $A$ as union of its finitely presented $\mathbb C$-subalgebras
	\[ A = \bigcup_{\alpha} A_\alpha . \]
	The collection of finitely presented $\mathbb C$-subalgebras of $A$ form a filtered category.
	Since both functors $(-)\an$ and $(-)\alg$ commute with filtered colimits, and since flat morphisms are stable under filtered colimits, it is enough to prove the statement when $A$ is a finitely presented $\mathbb C$-algebra.
	In this case, it follows from Lemmas \ref{lem:local_analytification_affine_space} and \ref{lem:analytification_discrete} that $(A\an)\alg$ is again discrete.
	The conclusion now follows from the classical result of J.\ P.\ Serre \cite[Corollaire 6.1]{Serre_GAGA}.
	
	We now prove (2) by induction on $n \ge 0$.
	We start by considering the case $n = 0$.
	In this case
	\[ \mathbb C[S^0] \simeq \mathbb C[X,Y] . \]
	Let $f \colon \mathbb C[X,Y] \to \mathbb C$ be an augmentation.
	Up to replacing $X$ and $Y$ by $X - f(X)$ and $Y - f(Y)$ respectively, we can assume $f(X) = f(Y) = 0$.
	In this case, \cref{lem:local_analytification_affine_space} identifies the ring $(\mathbb C[X, Y]\an)\alg$ with the ring of germs of holomorphic functions $\mathbb C\{X, Y\}$ at the origin of $\mathbf A^2_{\mathbb C}$.
	In particular, the flatness of
	\[ \mathbb C[X, Y] \to \mathbb C\{X, Y\} \]
	follows because both rings are noetherian and their formal completions at the maximal ideal of $\mathbb C\{X, Y\}$ coincide with the ring of formal power series $\mathbb C \llb X, Y \rrb$.
	Consider furthermore the map $\mathbb C[X,Y] \to \mathbb C[X]$ given by $X, Y \mapsto X$.
	Then the diagram
	\[ \begin{tikzcd}
		\mathbb C[X, Y] \arrow{r}{f} \arrow{d} & \mathbb C[X] \arrow{d} \\
		(\mathbb C[X, Y]\an)\alg \arrow{r} & (\mathbb C[X]\an)\alg ,
	\end{tikzcd} \]
	is a pushout in $\CRing_{\mathbb C}^{\mathrm{aug}}$.
	In particular, the morphism $\mathbb C[X, Y]\an \to \mathbb C[X]\an$ is an effective epimorphism in $\AnRing_{\mathbb C}$.
	
	Let now $\mathbb C[X, Y] \to A$ be any morphism in $\CRing_{\mathbb C}^{\mathrm{aug}}$ and set
	\[ B \coloneqq A \otimes_{\mathbb C[X, Y]} \mathbb C[X] . \]
	Consider the commutative cube
	\[ \begin{tikzcd} [column sep = small, row sep = small]
		{} & \mathbb C[X, Y] \arrow{rr} \arrow{dd} \arrow{dl} & & \mathbb C[X] \arrow{dd} \arrow{dl} \\
		A \arrow[crossing over]{rr} \arrow{dd} & & B \\
		{} & (\mathbb C[X, Y]\an)\alg \arrow{rr} \arrow{dl} & & (\mathbb C[X]\an)\alg \arrow{dl} \\
		(A\an)\alg \arrow{rr} & & (B\an)\alg . \arrow[leftarrow, crossing over]{uu}
	\end{tikzcd} \]
	We already remarked that the back square is a pushout.
	Notice then that since $(-)\an$ is a left adjoint, the square
	\[ \begin{tikzcd}
		\mathbb C[X, Y]\an \arrow{r} \arrow{d} & \mathbb C[X]\an \arrow{d} \\
		A\an \arrow{r} & B\an
	\end{tikzcd} \]
	is a pushout square in $\AnRing_{\mathbb C}$.
	The unramifiedness of $\cTan$ implies therefore that the bottom square of the previous commutative cube is a pushout as well (see \cite[Proposition 11.12(3)]{DAG-IX}).
	As a consequence, the outer and the left square in the rectangle
	\[ \begin{tikzcd}
		\mathbb C[X, Y] \arrow{r} \arrow{d} & A \arrow{r} \arrow{d} & B \arrow{d} \\
		(\mathbb C[X, Y]\an)\alg \arrow{r} & (A\an)\alg \arrow{r} & (B\an)\alg
	\end{tikzcd} \]
	are pushout squares in $\AnRing_{\mathbb C}$.
	It follows from \cite[4.4.2.1]{Lurie_Higher_algebra} that the same goes for the square on the right.
	Since flat maps are stable under base change, we see that statement (2) holds for $m = 0$.
	
	We now suppose $n > 0$ and we argue by induction.
	Recall that the functor $\mathbb C[-] \colon \cS \to \CRing_{\mathbb C}$ is a left adjoint.
	In particular, we have by definition a pushout diagram
	\[ \begin{tikzcd}
		\mathbb C[S^{n-1}] \arrow{r} \arrow{d} & \mathbb C[X] \arrow{d} \\
		\mathbb C[X] \arrow{r} & \mathbb C[S^n] .
	\end{tikzcd} \]
	The inductive hypothesis therefore guarantees that the map
	\[ \mathbb C[S^n] \longrightarrow ( \mathbb C[S^n]\an )\alg \]
	is flat.
	Furthermore, the induction hypothesis guarantees also that the square
	\[ \begin{tikzcd}
		\mathbb C[S^n] \arrow{r}{f} \arrow{d} & \mathbb C[X] \arrow{d} \\
		(\mathbb C[S^n]\an)\alg \arrow{r} & (\mathbb C[X]\an)\alg
	\end{tikzcd} \]
	is a pushout.
	Indeed, let
	\[ A \coloneqq (\mathbb C[S^n]\an)\alg \otimes_{\mathbb C[S^n]} \mathbb C[X] . \]
	Since the map $\mathbb C[S^n] \to (\mathbb C[S^n]\an)\alg$ is flat, the same goes for the map $\mathbb C[X] \to A$.
	In particular, it follows that $A$ is discrete.
	It is then enough to verify that
	\[ \pi_0(A) \simeq (\mathbb C[X]\an)\alg . \]
	This follows because
	\[ \pi_0( ( \mathbb C[S^n]\an )\alg ) \simeq (\pi_0( \mathbb C[S^n]\an ) )\alg \simeq ( ( \pi_0 \mathbb C[S^n] )\an )\alg \simeq ( \mathbb C[\pi_0(S^n)]\an )\alg . \]
	Indeed, since $n > 0$, $\pi_0(S^n)$ is reduced to a single point.
	Notice that the first equivalence above is a consequence of the compatibility of the pregeometry $\cTank$ with $0$-truncations (see \cite[Proposition 11.4]{DAG-IX}), the second equivalence follows combining Lemmas \ref{lem:local_analytification_affine_space} and \ref{lem:analytification_discrete} and \cref{prop:relative_spectrum_truncated_objects}.
	The last equivalence follows from the universal properties of the free algebra and the $\pi_0$ functor.
	
	Therefore, the same proof of the case $n = 0$ applies.
\end{proof}

\begin{cor} \label{cor:analytification_flat}
	Let $X \coloneqq (\cX, \cO_X)$ be a complex derived \DM stack locally almost of finite presentation.
	Let $X\an \coloneqq (\cX\an, \cO_{X\an}) \in \dAnc$ be its analytification.
	Then the canonical map
	\[ p \colon (\cX\an, \cO_{X\an}) \to (\cX, \cO_X) \]
	is flat.
\end{cor}

\begin{proof}
	The question is local on $\cX\an$ and can therefore be tested on the points of $\cX\an$.
	If $x_* \colon \cS \leftrightarrows \cX\an \colon x\inv$ is a geometric point, then $x\inv \cO_{X\an} \in \AnRing_{\mathbb C}$ and $x\inv p\inv \cO_X \in \CRing_{\mathbb C}^{\mathrm{aug}}$.
	Furthermore, \cref{lem:local_analytification_affine_space} shows that
	\[ x\inv \cO_{X\an} \simeq ( x\inv p\inv \cO_X )\an . \]
	Therefore, the conclusion follows from \cref{thm:local_analytification_flat}.
\end{proof}

\section{Grauert theorem for derived complex analytic stacks} \label{sec:Grauert}

In this section, we deduce from \cite[Theorem 5.20]{Porta_Yu_Higher_analytic_stacks_2014} the Grauert theorem for proper morphism of derived analytic Artin stacks.
Let us start by introducing the objects we are interested in. \\

Denote by $\bP_\mathrm{sm}$ the collection of smooth morphisms of derived analytic spaces (see \cref{def:differential_properties}).
Then $(\dStn, \tauan, \bP_\mathrm{sm})$ forms a geometric context in the sense of \cite[Definition 2.2]{Porta_Yu_Higher_analytic_stacks_2014}.

\begin{defin}
	A derived analytic Artin stack is a geometric stack with respect to the context $(\dStn, \tauan, \bP_\mathrm{sm})$ in the sense of \cite[Definition 2.8]{Porta_Yu_Higher_analytic_stacks_2014}.
\end{defin}

Notice that the inclusion
\[ j \colon (\Stn, \tauan, \bP_\mathrm{sm}) \to (\dStn, \tauan, \bP_\mathrm{sm}) \]
is a morphism of geometric contexts (in the sense of \cite[Definition 2.23]{Porta_Yu_Higher_analytic_stacks_2014}).
Observe furthermore that if $X = (\cX, \cO_X) \in \dStn$ is a derived Stein space and $F_X = \phi(X)$ is its functor of points, then
\[ j^s(F_X) \coloneqq F_X \circ j \]
coincides with the functor of points of the ordinary Stein space $\trunc(X) \coloneqq (\cX,\pi_0(\cO_X))$.
It follows from \cite[Proposition 2.25]{Porta_Yu_Higher_analytic_stacks_2014} that the restriction functor
\[ j^s \colon \St(\dStn, \tauan) \to \St(\Stn, \tauan) \]
takes geometric stacks to geometric stacks.
As $j_s$ does the same in virtue of \cite[Lemma 2.24]{Porta_Yu_Higher_analytic_stacks_2014}, we conclude that the functor
\[ j_s j^s \colon \St(\dStn, \tauan) \to \St(\dStn, \tauan) \]
takes geometric stacks to geometric stacks.
We refer to this functor as the \emph{truncation functor}, and we denote it by $\trunc$.

\begin{defin} \label{def:proper_morphism}
	Let $f \colon X \to Y$ be a morphism of derived analytic Artin stacks.
	We say that $f$ is \emph{proper} if its truncation $\trunc(f) \colon \trunc(X) \to \trunc(Y)$ is proper in the sense of \cite[Definition 4.8]{Porta_Yu_Higher_analytic_stacks_2014}.
\end{defin}

Finally, we introduce the $\infty$-category of coherent sheaves on a derived analytic Artin stack.
We start with the case of a derived analytic space.

\begin{defin}
	Let $X = (\cX, \cO_X)$ be a derived analytic space.
	We define
	\[ \cO_X \Mod \coloneqq \cO_X\alg \Mod , \]
	where the latter is the stable $\infty$-category of $\DAb$-valued sheaves on $\cX$ equipped with a $\cO_X\alg$-module structure.
	We equip $\cO_X \Mod$ with the standard $t$-structure introduced in \cite[Proposition 2.1.3]{DAG-VIII}.
	We denote by $\Coh(X)$ the full subcategory of $\cO_X \Mod$ spanned by those $\cO_X$-modules $\cF$ such that each $\pi_i(\cF)$ is a coherent sheaf over $\pi_0(\cO_X\alg)$.
\end{defin}

Notice that $\Coh(X)$ is a stable subcategory of $\cO_X \Mod$, and that the $t$-structure on $\cO_X\Mod$ restricts to a $t$-structure on $\Coh(X)$.
In particular, the notations $\Coh^-(X)$, $\Coh^+(X)$ and $\Cohb(X)$ are defined.
Following \cite{Lurie_Higher_algebra}, we use \emph{homological} convention.
The next result is elementary but of fundamental importance:

\begin{lem} \label{lem:coh_heart}
	Let $X$ be a derived analytic Artin stack.
	The canonical map $j \colon \trunc(X) \to X$ induces a $t$-exact functor
	\[ j_* \colon \Coh(\trunc(X)) \to \Coh(X) . \]
	Furthermore, it induces an equivalence
	\[ j_* \colon \Cohh(\trunc(X)) \simeq \Cohh(X) . \]
\end{lem}

\begin{proof}
	Notice that the $t$-structure on $\Coh(X)$ allows to prove the smooth descent property of $X \mapsto \Coh(X)$ exactly as in the same way of the algebraic setting.\footnote{See \cite[\S 6]{DAG-VII}.}
	In particular, we are reduced to prove this lemma when $X$ is a derived analytic space.
	In this case, it is a direct consequence of \cite[Remark 2.1.5]{DAG-VIII}.
\end{proof}

We can extend the definition of $\Coh(X)$ to a derived analytic Artin stack exactly in the same way it is done for underived analytic Artin stacks.
We refer to \cite[\S 5.1]{Porta_Yu_Higher_analytic_stacks_2014} for a detailed explanation, and we do not repeat it here.

We finally have:

\begin{thm} \label{thm:Grauert_theorem}
	Let $f \colon X \to Y$ be a proper morphism of derived analytic Artin stacks.
	Then the derived pushforward
	\[ f_* \colon \cO_X \Mod \to \cO_Y \Mod \]
	restricts to a functor
	\[ f_* \colon \Coh^-(X) \to \Coh^-(Y) . \]
\end{thm}

\begin{proof}
	Let $\cC$ be the full subcategory of $\Cohb(X)$ spanned by those $\cF$ such that $f_*(\cF) \in \Cohb(Y)$.
	Notice that:
	\begin{enumerate}
		\item $\cC$ is closed under loops, suspensions and extensions. This is because $f_*$ is an exact functor of stable $\infty$-categories.
		\item $\cC$ contains $\Cohh(X)$.
		Indeed, consider the commutative diagram
		\[ \begin{tikzcd}
			\trunc(X) \arrow{r}{\trunc(f)} \arrow{d}{i} & \trunc(Y) \arrow{d}{j} \\
			X \arrow{r}{f} & Y .
		\end{tikzcd} \]
		In virtue of \cref{lem:coh_heart}, we know that $i_*$ and $j_*$ are $t$-exact and that they induce equivalences on the hearts of the $t$-structures.
		Given $\cF \in \Cohh(X)$, we can therefore write $\cF \simeq i_*(\cG)$, where $\cG \in \Cohh(\trunc(X))$.
		As a consequence, we obtain
		\[ f_*(\cF) \simeq f_*( i_*(\cG)) \simeq j_*( \trunc(f)_*(\cG) ) . \]
		Since $j_*$ is $t$-exact and respect coherent sheaves, we are reduced to prove that $\trunc(f)_*(\cG) \in \Coh^-(Y)$.
		This follows directly from \cite[Theorem 5.20]{Porta_Yu_Higher_analytic_stacks_2014}.
	\end{enumerate}
	These two points together imply that $\Cohb(X) \subseteq \cC$.
	To complete the proof, it is now enough to prove that an object $\cF \in \Coh^-(X)$ belongs to $\cC$ if and only if $\tau_{\le n}\cF$ belongs to $\cC$ for every $n \le 0$.
	Consider the fiber sequence
	\[ \tau_{\ge n+1} \cF \to \cF \to \tau_{\le n} \cF .  \]
	Applying $f_*$ yields the fiber sequence
	\[ f_*( \tau_{\ge n+1} \cF) \to f_* (\cF) \to f_*( \tau_{\le n} \cF) . \]
	Since $f_*$ is right $t$-exact, we see that $f_*(\tau_{\le n} \cF) \in \Coh^{\le n}(X)$.
	As a consequence, there are isomorphisms
	\[ \rH^i( f_*(\tau_{\ge n+1} \cF )) \simeq \rH^i( f_*(\cF) ) \]
	for every $i \ge n + 1$. Since $\tau_{\ge n+1} \cF \in \Cohb(X)$ and since this argument works for every $n \le 0$, the conclusion follows.
\end{proof}

\section{GAGA theorems} \label{sec:GAGA}

We now turn to the main results of this paper.

\begin{thm} \label{thm:GAGA1}
	Let $f \colon X \to Y$ be a proper morphism of derived complex Artin stacks locally almost of finite presentation.\footnote{Just as in \cref{def:proper_morphism}, a morphism of derived Artin stacks $f$ is said to be proper if its truncation $\trunc(f)$ is proper (see for example \cite[Definition 4.8]{Porta_Yu_Higher_analytic_stacks_2014}). It follows from \cite[Proposition 6.3]{Porta_Yu_Higher_analytic_stacks_2014} that proper morphisms are preserved under analytification.}
	Then the diagram
	\[ \begin{tikzcd}
		\Coh^-(X) \arrow{r}{(-)\an} \arrow{d}{f_*} & \Coh^-(X\an) \arrow{d}{f\an_*} \\
		\Coh^-(Y) \arrow{r}{(-)\an} & \Coh^-(Y\an)
	\end{tikzcd} \]
	commutes.
\end{thm}

\begin{proof}
	Since the diagram
	\[ \begin{tikzcd}
		\Coh^-(X) \arrow{r}{(-)\an} & \Coh^-(X\an) \\
		\Coh^-(Y) \arrow{r}{(-)\an} \arrow{u}[swap]{f^*} & \Coh^-(Y\an) \arrow{u}[swap]{(f\an)^*}
	\end{tikzcd} \]
	commutes, there is for every $\cF \in \Coh^-(X)$ a natural morphism
	\[ \eta_\cF \colon (f_*(\cF))\an \to f\an_*(\cF\an) . \]
	Let $\cC$ be the full subcategory of $\Coh^-(X)$ spanned by those $\cF$ such that $\eta_\cF$ is an equivalence.
	Since $(-)\an$, $f_*$ and $f\an_*$ are exact functors of stable $\infty$-categories, we see that $\cC$ is a stable subcategory of $\Coh^-(X)$.
	Furthermore, reasoning as in the proof of \cref{thm:Grauert_theorem}, we can use \cref{lem:coh_heart} and \cite[Theorem 7.1]{Porta_Yu_Higher_analytic_stacks_2014} to conclude that $\cC$ contains $\Cohh(X)$.
	It follows that $\cC$ contains $\Cohb(X)$ as well.
	Using the smooth descent of $\Coh^-(X)$ and the flatness result provided by \cref{cor:analytification_flat}, we conclude that the functor
	\[ (-)\an \colon \Coh^-(X) \to \Coh^-(X\an) \]
	is $t$-exact.
	We use this fact to prove that an $\cC$ coincides with $\Coh^-(X)$.
	Fix $\cF \in \Coh^-(X)$ and consider the fiber sequence
	\[ \tau_{\ge n} \cF \longrightarrow \cF \longrightarrow \tau_{\le n - 1} \cF . \]
	It induces a morphism of fiber sequences
	\[ \begin{tikzcd}
		( f_*( \tau_{\ge n} \cF ) )\an \arrow{r} \arrow{d} & ( f_* (\cF) )\an \arrow{r} \arrow{d} & ( f_*( \tau_{\le n - 1} \cF ) )\an \arrow{d} \\
		f\an_*( \tau_{\ge n} \cF\an ) \arrow{r} & f\an_*(\cF\an) \arrow{r} & f\an_*(\tau_{\le n - 1} \cF\an) .
	\end{tikzcd} \]
	Since $(-)\an$ is $t$-exact, we see that both $( f_*( \tau_{\le n-1} \cF) )\an$ and $f\an_*(\tau_{\le n-1} \cF\an)$ belong to $\Coh^{\le n-1}(X)$.
	As a consequence, we deduce that there are isomorphisms
	\[ \pi_i( ( f_*( \tau_{\ge n +1} \cF ) )\an ) \simeq \pi_i( ( f_* (\cF) )\an ), \qquad \pi_i( f\an_*( \tau_{\ge n+1} \cF\an ) ) \simeq \pi_i(f\an_*(\cF\an)) \]
	for every $i \ge n$.
	Since $\tau_{\ge n} \cF$ belongs to $\Cohb(X)$, and hence to $\cC$, we conclude that $\eta_\cF$ induces an isomorphism
	\[ \pi_i( f_*(\cF)\an ) \simeq \pi_i( f\an_*(\cF\an) ) \]
	for every $i \ge n$.
	Repeating this argument for every $n \le 0$, we finally conclude that $\cF \in \cC$, thus achieving the proof.
\end{proof}

\begin{thm} \label{thm:GAGA2}
	Let $X$ be a proper derived Artin stack locally almost of finite presentation over $\mathbb C$.
	Then the analytification functor induces an equivalence
	\[ \Coh(X) \simeq \Coh(X\an) . \]
\end{thm}

\begin{proof}
	We start by proving fully faithfulness.
	For any pair of objects $\cF, \cG \in \Coh(X)$, we let
	\[ \psi_{\cF, \cG} \colon \Map_{\Coh(X)}(\cF, \cG) \to \Map_{\Coh(X\an)}(\cF\an, \cG\an) \]
	be the natural map.
	
	We start by proving that $\psi_{\cF, \cG}$ is an equivalence whenever $\cF, \cG \in \Cohb(X)$.
	Let $\cHom_X(\cF, \cG)$ be the internal hom in $\cO_X \Mod$.
	Notice that
	\[ \cHom_X(\cF, \cG)\an \simeq \cHom_{X\an}(\cF\an, \cG\an) . \]
	Furthermore, if $\cF, \cG \in \Cohb(X)$, then
	\[ \cHom_X(\cF, \cG) \in \Coh^-(X) . \]
	Finally, observe that
	\[ \Map_{\Coh(X)}(\cF, \cG) \simeq \tau_{\ge 0} p_* \cHom_X(\cF, \cG), \quad \Map_{\Coh(X\an)}(\cF\an, \cG\an) \simeq \tau_{\ge 0} p\an_* \cHom_{X\an}(\cF\an, \cG\an) , \]
	where $p \colon X \to \Spec(\mathbb C)$ is the canonical morphism.
	Therefore, \cref{thm:GAGA1} implies that $\psi_{\cF, \cG}$ is an equivalence in this case.
	
	Now let us fix two general $\cF, \cG \in \Coh(X)$.
	Observe that
	\[ \cG \simeq \lim_{n \ge 0} \tau_{\le n} \cG . \]
	Since $(-)\an$ is $t$-exact in virtue of \cref{cor:analytification_flat}, we have a canonical equivalence
	\[ \cG\an \simeq \lim_{n \ge 0} \tau_{\le n}(\cG\an) \simeq \lim_{n \ge 0} (\tau_{\le n} \cG)\an . \]
	Consider the following commutative square
	\[ \begin{tikzcd}
		\Map_{\Coh(X)}(\cF, \cG) \arrow{d}{\psi_{\cF, \cG}} \arrow{r} & \lim_{n \ge 0} \Map_{\Coh(X)}(\cF, \tau_{\le n} \cG) \arrow{d}{\psi_{\cF, \tau_{\le n} \cG}} \\
		\Map_{\Coh(X\an)}(\cF\an, \cG\an) \arrow{r} & \lim_{n \ge 0} \Map_{\Coh(X\an)}(\cF\an, \tau_{\le n} \cG\an) .
	\end{tikzcd} \]
	As the horizontal morphisms are equivalences, we conclude that it is enough to check that the morphisms $\psi_{\cF, \tau_{\ge n} \cG}$ is an equivalence for every $n$.
	In other words, we can assume without loss of generality that $\cG \in \Coh^-(X)$.
	On the other hand, we can write
	\[ \cF \simeq \colim_{n \le 0} \tau_{\ge n} \cF , \quad \cF\an \simeq \colim_{n \le 0} (\tau_{\ge n} \cF)\an . \]
	The same argument given above implies then that we can assume without loss of generality that $\cF \in \Coh^+(X)$.
	Suppose $\cF \in \Coh^{\ge n}(X)$ and $\cG \in \Coh^{\le m}(X)$.
	Then using \cite[1.2.1.5]{Lurie_Higher_algebra} we have
	\[ \Map_{\Coh(X)}(\cF, \cG) \simeq \Map_{\Coh(X)}( \tau_{\le m} \cF, \cG ) \simeq \Map_{\Coh(X)}(\tau_{\le m} \cF, \tau_{\ge n} \cG) \]
	and, similarly,
	\[ \Map_{\Coh(X\an)}(\cF\an, \cG\an) \simeq \Map_{\Coh(X\an)}( \tau_{\le m} \cF\an, \cG\an ) \simeq \Map_{\Coh(X)}(\tau_{\le m} \cF\an, \tau_{\ge n} \cG\an) . \]
	Since both $\tau_{\le m} \cF$ and $\tau_{\ge n} \cG$ belongs to $\Cohb(X)$, the first part of the proof implies that the map
	\[ \Map_{\Coh(X)}(\tau_{\le m} \cF, \tau_{\ge n} \cG) \to \Map_{\Coh(X)}(\tau_{\le m} \cF\an, \tau_{\ge n} \cG\an) \]
	is an equivalence.
	As a consequence, $\psi_{\cF, \cG}$ is an equivalence as well, thus completing the proof of the fully faithfulness of $(-)\an$. \\
	
	We now prove that $(-)\an$ is essentially surjective.
	Let us denote by $\cC$ the essential image of $(-)\an$.
	Since $(-)\an$ is fully faithful, $\cC$ is a stable subcategory of $\Coh(X\an)$.
	Indeed, given a fiber sequence
	\[ \cF_1 \xrightarrow{\alpha_1} \cF_2 \xrightarrow{\alpha_2} \cF_3 \]
	such that any two sheaves belong to $\cC$, then so does the third one.
	To see this, up to rotating the fiber sequence, we can assume $\cF_2, \cF_3 \in \cC$.
	We can therefore find $\cG_2, \cG_3 \in \Coh(X)$ such that $\cG_2\an \simeq \cF_2$ and $\cG_3\an \simeq \cF_3$.
	Since $(-)\an$ is fully faithful, we can find $\beta_2 \colon \cG_2\to \cG_3$ such that $\beta_2\an \simeq \alpha_2$.
	Set
	\[ \cG_1 \coloneqq \fib\left( \cG_2 \xrightarrow{\beta_2} \cG_3 \right) . \]
	Since $(-)\an$ is exact, we obtain a natural equivalence $\cG_1\an \simeq \cF_1$.
	In other words, $\cF_1 \in \cC$.
	
	We now observe that $\cC$ contains $\Cohh(X\an)$.
	Indeed, fix $\cF \in \Cohh(X\an)$.
	Let $j \colon \trunc(X) \to X$ be the canonical morphism and let $j\an \colon \trunc(X)\an \to X\an$ be its analytification.
	Combining Propositions \ref{prop:relative_spectrum_and_truncations} and \ref{prop:relative_spectrum_truncated_objects} we can identify $j\an$ with the canonical morphism
	\[ \trunc(X\an) \longrightarrow X\an . \]
	Under the equivalence
	\[ \Cohh(X\an) \simeq \Cohh(\trunc(X\an)) \]
	guaranteed by \cref{lem:coh_heart}, we can find $\widetilde{\cF} \in \Cohh(\trunc(X\an))$ such that $j\an_* ( \widetilde{\cF} ) \simeq \cF$.
	We can now invoke \cite[Theorem 7.4]{Porta_Yu_Higher_analytic_stacks_2014} to deduce the existence of $\cG \in \Coh(\trunc(X))$ such that $\cG\an \simeq \widetilde{\cF}$.
	Therefore, \cref{thm:GAGA1} implies that
	\[ j_*(\cG)\an \simeq j\an_*(\cG\an) \simeq j\an_* ( \widetilde{\cF} ) \simeq \cF . \]
	As a consequence, we deduce that $\cC$ contains $\Cohb(X\an)$.
	
	Let now $\cF \in \Coh^+(X\an)$.
	We can write
	\[ \cF \simeq \lim_{n \ge 0} \tau_{\le n} \cF , \]
	and for each $n \ge 0$, $\tau_{\le n} \cF \in \Cohb(X\an)$.
	We can therefore find $\cG_n \in \Coh(X)$ such that $\cG_n\an \simeq \tau_{\le n} \cF$.
	The fully faithfulness of $(-)\an$ allows to find maps $i_n \colon \cG_n \to \cG_{n-1}$ for every $n$ such that $i_n\an$ is equivalent to the map $\tau_{\le n} \cF \to \tau_{\le n-1} \cF$.
	Furthermore, fully faithfulness of $(-)\an$ allows to assemble these maps in a diagram $G \colon \mathbb Z_{\le 0} \to \Coh(X)$.
	Since $(-)\an$ is $t$-exact and conservative, we conclude that $i_n$ induces an equivalence
	\[ \tau_{\le n-1} \cG_n \simeq \cG_{n-1} \]
	for every $n$.
	Since the $t$-structure on $\Coh(X)$ is complete, we conclude that the limit $\cG$ of the diagram $G$ exists and that it satisfies $\tau_{\le n} \cG \simeq \cG_n$.
	As a consequence, we obtain
	\[ \cG\an \simeq \lim_{n \ge 0} ( \tau_{\le n} \cG )\an \simeq \lim_{n \ge 0} \tau_{\le n} \cF \simeq \cF . \]
	Therefore $\cC$ contains $\Coh^+(X\an)$.
	A dual argument shows that it also contains $\Coh^-(X\an)$, thus completing the proof.
\end{proof}

\section{Applications} \label{sec:applications}

We now provide two applications of \cref{thm:GAGA2}: on one hand, we show that a derived analytic space is algebraizable if and only if its truncation is.
On the other hand, we prove that the local structure of a (sufficiently nice) analytic moduli problem is governed by a differential graded Lie algebra, thus extending the main result of \cite{DAG-X} to the analytic setting.

Both these applications use our GAGA theorems in an essential way.
However, they also rely on other properties of derived analytic spaces.
Most notably, we will use the analytic cotangent complex and the Postnikov tower decomposition.
Both these tools have been introduced in the author's Ph.D.\ thesis and can be found in \cite{Porta_Yu_Representability}.

\subsection{Algebraizability of derived complex analytic spaces}

Our first application is the following:

\begin{prop} \label{prop:algebraizable}
	Let $X = (\cX, \cO_X)$ be a proper derived analytic space.
	Then $X$ is algebraizable if and only if $\trunc(X)$ is algebraizable.
\end{prop}

\begin{proof}
	Suppose first that there exists a derived \DM stack $Y$ locally almost of finite presentation over $\Spec(\mathbb C)$ such that $Y\an \simeq X$.
	\Cref{prop:relative_spectrum_and_truncations} implies that
	\[ \trunc(X) \simeq \trunc(Y\an) \simeq \trunc(Y)\an , \]
	and therefore $\trunc(X)$ is algebraizable.
	
	Suppose now that $\trunc(X) \simeq Y_0\an$ for some derived \DM stack $Y_0 = (\cY, \cO_{Y_0})$.
	We start by constructing inductively on $n$ a sequence of sheaves of derived rings
	\[ \cdots \to \cO_{Y_{n+1}} \to \cO_{Y_n} \to \cdots \to \cO_{Y_0} \]
	such that
	\begin{enumerate}
		\item the morphism $\cO_{Y_{n+1}} \to \cO_{Y_n}$ induces an equivalence $\tau_{\le n} \cO_{Y_{n+1}} \simeq \cO_{Y_n}$;
		\item set $Y_n \coloneqq (\cY, \cO_{Y_n})$. Then $Y_n\an \simeq \mathrm t_{\le n} X \coloneqq (\cX, \tau_{\le n}\cO_X)$.
	\end{enumerate}
	Suppose first that this sequence has already been built.
	Let
	\[ \cO_Y \coloneqq \lim_{n \ge 0} \cO_{Y_n} , \]
	and set $Y \coloneqq (\cY, \cO_Y)$.
	Then condition (1) guarantees that $\tau_{\le n} \cO_Y \simeq \cO_{Y_n}$.
	Therefore \cref{prop:relative_spectrum_and_truncations} implies that
	\[ \mathrm t_{\le n} (Y\an) \simeq (\mathrm t_{\le n} Y)\an \simeq \mathrm t_{\le n} X . \]
	Since derived analytic \DM stacks are convergent (see \cite[Lemma 7.7]{Porta_Yu_Representability}), we obtain in this way a map $Y\an \to X$ that induces an equivalence on each truncation. It follows that this map is an equivalence, thus showing that $X$ is algebraizable.

	We are left to build the aforementioned sequence.
	When $n = 0$, condition (1) is empty and condition (2) holds by assumption.
	Suppose now that $\cO_{Y_n}$ has already been constructed in such a way that both conditions hold.
	Using \cite[Corollary 5.42]{Porta_Yu_Representability} we can find an analytic derivation
	\[ d \colon \anL_{\mathrm t_{\le n} X} \to \pi_{n + 1}(\cO_X\alg)[n+2] \]
	such that the square
	\[ \begin{tikzcd}
		\tau_{\le n+1} \cO_X \arrow{r} \arrow{d} & \tau_{\le n} \cO_X \arrow{d}{\eta_d} \\
		\tau_{\le n} \cO_X \arrow{r}{\eta_0} & \tau_{\le n} \cO_X \oplus \pi_{n+1}(\cO_X\alg)[n+2]
	\end{tikzcd} \]
	is a pullback square in $\AnRing_{\mathbb C}(\cX)$, where $\eta_d$ classifies $d$ and $\eta_0$ classifies the zero derivation.
	Set $\cF \coloneqq \pi_{n+1}(\cO_X\alg)[n+2]$ and observe that it is a coherent sheaf over $\mathrm t_{\le n} X$.
	We can therefore invoke \cref{thm:GAGA2} to deduce the existence of a coherent sheaf $\cG \in \Coh(Y_n)$ such that $\cG\an \simeq \cF$.
	Furthermore, since $(-)\an$ is $t$-exact and conservative, we deduce that $\cG$ is concentrated in homological degree $n+2$.
	We now use \cite[Theorem 5.20]{Porta_Yu_Representability} to deduce the equivalence
	\[ \mathbb L_{Y_n}\an \simeq \anL_{\mathrm t_{\le n} X} . \]
	We can therefore use \cref{thm:GAGA2} once more to deduce the existence of an algebraic derivation
	\[ \delta \colon \mathbb L_{Y_n} \to \cG \]
	such that $\delta\an \simeq d$.
	We now define $\cO_{Y_{n+1}}$ to be the pullback
	\[ \begin{tikzcd}
		\cO_{Y_{n+1}} \arrow{r} \arrow{d} & \cO_{Y_n} \arrow{d}{\eta_{\delta}} \\
		\cO_{Y_n} \arrow{r}{\eta_0} & \cO_{Y_n} \oplus \cG .
	\end{tikzcd} \]
	We are left to check that $Y_{n+1}\an \simeq \mathrm t_{\le n+1} X$.
	We claim first that the analytification functor
	\[ (-)\an \colon \CRing_{\mathbb C}(\cY)_{/\cO_{Y_n}} \to \AnRing_{\mathbb C}(\cX)_{/\tau_{\le n} \cO_X} \]
	takes $\cO_{Y_n} \oplus \cG$ to the analytic split square-zero extension $\tau_{\le n} \cO_X \oplus \cF$.
	To see this, we observe that \cite[Lemma 5.15]{Porta_Yu_Representability} produces a canonical morphism
	\[ \alpha \colon ( \cO_{Y_n} \oplus \cG )\an \to \tau_{\le n} \cO_X \oplus \cF . \]
	Since the underlying algebra functor $(-)\alg$ is conservative, it is enough to check that $\alpha\alg$ is an equivalence.
	Notice that we have a commutative triangle
	\[ \begin{tikzcd}
		{} & \cO_{Y_n} \oplus \cG \arrow{dl} \arrow{dr} \\
		((\cO_{Y_n} \oplus \cG)\an)\alg \arrow{rr}{\alpha\alg} & & \tau_{\le n} \cO_X\alg \oplus \cF .
	\end{tikzcd} \]
	The left diagonal map is flat (in virtue of \cref{cor:analytification_flat}), and the right diagonal map is strong.
	It follows that $\alpha\alg$ is strong.
	Since $\cG$ and $\cF$ are in homological degree $n + 2$ and $\cO_{Y_n}\an \simeq \tau_{\le n} \cO_X$, we conclude that $\pi_0(\alpha\alg)$ is an isomorphism.
	Hence $\alpha\alg$ is an isomorphism as well.
	This completes the proof of the claim.
	At this point, the conclusion follows from the fact that both algebraic and analytic derived \DM stacks are infinitesimally cartesian (see \cite{DAG-XIV} for the algebraic case and \cite[Lemma 7.7]{Porta_Yu_Representability} for the analytic case).
\end{proof}

\subsection{Analytic formal moduli problems}

Let $F \colon \dStn\op \to \cS$ be a functor.
We show that, as in the algebraic case, the formal geometry around a point $x \in F(*_{\mathbb C})$ is governed by a differential graded Lie algebra.
This is an analytic analogue of the formal moduli problem correspondence in derived algebraic geometry, first proven by J.\ Pridham \cite{Pridham_Unifying} and subsequently generalized by J.\ Lurie \cite{DAG-X}. \\

We work in the framework of axiomatic deformation theories introduced in \cite{DAG-X}.
After introducing the notion of analytic artinian object, we deduce from our GAGA theorem that the $\infty$-category of analytic artinian objects is equivalent to the one of algebraic artinian objects. The conclusion follows then from \cite[Theorem 2.0.2]{DAG-X}.

We define a deformation context as follows:
\begin{enumerate}
	\item the underlying $\infty$-category is $\AnRing_{\mathbb C}$;
	\item as family of spectra in $\AnRing_{\mathbb C}$ we take the one element family
	\[ \mathbb C \in \mathrm{Mod}_{\mathbb C} \simeq \Sp(\AnRing_{\mathbb C}) . \]
	Notice that this equivalence is consequence of \cite[Theorem 4.11]{Porta_Yu_Representability}.
\end{enumerate}
We denote by $\mathrm{AnArt}_{\mathbb C}$ the $\infty$-category of small objects with respect to this deformation context (see \cite[Definition 1.1.8]{DAG-X}).
We refer to $\mathrm{AnArt}_{\mathbb C}$ as the \emph{$\infty$-category of analytic Artinian objects}.
Our second application of the GAGA theorems is then the following:

\begin{prop}
	The underlying algebra functor $(-)\alg \colon \AnRing_{\mathbb C} \to \CRing_{\mathbb C}^{\mathrm{aug}}$ restricts to an equivalence of $\infty$-categories
	\[ \mathrm{AnArt}_{\mathbb C} \simeq \mathrm{Art}_{\mathbb C} . \]
\end{prop}

\begin{proof}
	We start by remarking that \cite[Lemma 5.15]{Porta_Yu_Representability} and the fact that $(-)\alg$ commutes with limits imply that $(-)\alg$ restricts to a functor
	\[ (-)\alg \colon \mathrm{AnArt}_{\mathbb C} \longrightarrow \mathrm{Art}_{\mathbb C} , \]
	which is conservative.
	Notice that \cite[Theorem 6.5]{Porta_Yu_Representability} implies that derived analytic spaces are cartesian.
	As a consequence, we deduce that the analytification functor
	\[ (-)\an \colon \CRing_{\mathbb C}^{\mathrm{aug}} \to \AnRing_{\mathbb C} \]
	commutes with pullbacks
	\[ \begin{tikzcd}
		A' \arrow{r} \arrow{d} & B' \arrow{d}{f} \\
		A \arrow{r}{g} & B
	\end{tikzcd} \]
	where both $f$ and $g$ are surjective on $\pi_0$.
	In particular, $(-)\an$ restricts to a functor
	\[ (-)\an \colon \mathrm{Art}_{\mathbb C} \longrightarrow \mathrm{AnArt}_{\mathbb C} . \]
	Notice that \cref{prop:algebraizable} implies that $(-)\an$ is essentially surjective.
	To complete the proof it is then enough to check that for every $A \in \mathrm{Art}_{\mathbb C}$ the unit
	\[ \eta \colon A \longrightarrow (A\an)\alg \]
	is an equivalence.
	We know from \cref{thm:local_analytification_flat} that this map is flat.
	It is therefore enough to check that $\pi_0(\eta)$ is an equivalence.
	Since both $\pi_0(A)$ and $\pi_0((A\an)\alg)$ are Artinian, they are isomorphic to their formal completions, and hence they are isomorphic to each other.
\end{proof}

\begin{cor} \label{cor:analytic_FMP}
	Let $X$ be a derived analytic space and let $x \in X$ be a point.
	Then the shifted tangent space $\mathbb T_x X[-1]$ admits a canonical dg Lie algebra structure that completely determines the formal geometry of $X$ around $x$.
\end{cor}

\section{Appendix: complements on pregeometries} \label{sec:appendix}

We collect in this appendix a certain number of results that are used throughout the main body of the paper and that we could not locate in the literature.

\subsection{Local structures} \label{subsec:local_structures}

The category of local rings is not presentable. Indeed, it is not closed under products, nor under coproducts.
The situation changes if we restrict ourselves to the category of local rings with fixed residue field.

A similar phenomenon happens for a general pregeometry: if $\cT$ is a pregeometry and $\cX$ is an $\infty$-topos, then $\Strloc_\cT(\cX)$ is not presentable $\infty$-category, but for any fixed $\cO \in \Strloc_\cT(\cX)$, the $\infty$-category $\Strloc_\cT(\cX)_{/\cO}$ \emph{is} presentable.

This is a useful fact that plays a significant role in developing derived analytic geometry.
Our current goal is to provide a proof of this fact and to deduce some consequences.

\begin{defin}
	Let $\cT$ be a pregeometry.
	We denote by $\cT_{\mathrm{d}}$ the pregeometry having the same underlying $\infty$-category of $\cT$, whose admissible morphisms are isomorphisms and whose topology is the trivial one. We refer to $\cT_{\mathrm d}$ as the \emph{discrete pregeometry underlying $\cT$}.
	
	We denote by $\cT_{\mathrm{sd}}$ the pregeometry having the same underlying $\infty$-category and the same admissible morphisms of $\cT$, and whose topology is the trivial one.
	We refer to $\cT_{\mathrm{sd}}$ as the \emph{semi-discrete pregeometry underlying $\cT$}.
\end{defin}

Let $\cT$ be a pregeometry.
The identity of $\cT$ induces a transformation of pregeometries
\[ \varphi \colon \cT_{\mathrm{sd}} \to \cT . \]
Fix an $\infty$-topos $\cX$.
Composition with $\varphi$ induces a well defined functor
\[ \Strloc_\cT(\cX) \to \Strloc_{\cT_{\mathrm{sd}}}(\cX) , \]
which is not an equivalence because in general it is not essentially surjective.
Nevertheless, we have the following result:

\begin{prop} \label{prop:semi_discrete_simplification}
	Let $\cO \in \Strloc_\cT(\cX)$.
	Then the induced functor
	\[ \varphi_* \colon \Strloc_\cT(\cX)_{/\cO} \to \Strloc_{\cT_{\mathrm{sd}}}(\cX)_{/\cO\circ \varphi} \]
	is an equivalence of $\infty$-categories.
\end{prop}

\begin{proof}
	Committing a slight abuse of notation, we denote by $\cT$ the underlying $\infty$-category of both $\cT$ and $\cT_{\mathrm{sd}}$.
	Observe that there are natural functors
	\[ i \colon \Strloc_\cT(\cX)_{/\cO} \to \Fun(\cT, \cX)_{/\cO} , \quad i_{\mathrm{sd}} \colon \Strloc_{\cT_{\mathrm{sd}}}(\cX)_{/\cO \circ \varphi} \to \Fun_\cO(\cT, \cX)_{/\cO} \]
	that are compatible with $\varphi_*$ in the sense that the triangle
	\[ \begin{tikzcd}
		\Strloc_\cT(\cX)_{/\cO} \arrow{rr}{\varphi_*} \arrow{dr}[swap]{i} & & \Strloc_{\cT_{\mathrm{sd}}}(\cX)_{/\cO \circ \varphi} \arrow{dl}{i_{\mathrm{sd}}} \\
		{} & \Fun(\cT, \cX)_{/\cO}
	\end{tikzcd} \]
	commutes.
	Observe that $\Fun(\cT, \cX)_{/\cO}$ fits in the following pullback diagram of $\infty$-categories:
	\[ \begin{tikzcd}
		\Fun(\cT, \cX)_{/\cO} \arrow{r} \arrow{d} & \Fun(\cT \times \Delta^1, \cX) \arrow{d}{\ev_1} \\
		\{\cO\} \arrow{r} & \Fun(\cT, \cX) ,
	\end{tikzcd} \]
	where $\ev_1$ is induced by composition with $\cT \times \{1\} \hookrightarrow \cT \times \Delta^1$.
	Recall from \cite[2.4.7.12]{HTT} that $\ev_1$ is a Cartesian fibration.
	We can therefore use the fiber sequence provided by \cite[2.4.4.2]{HTT} to conclude that the functors $i$ and $i_{\mathrm{sd}}$ are fully faithful.
	It follows that $\varphi_*$ is fully faithful as well.
	
	We are therefore left to prove that $\varphi_*$ is essentially surjective.
	Let $\alpha \colon \cO' \to \cO$ be an object in $\Strloc_\cT(\cX)_{/\cO}$.
	Unraveling the definitions, we see that it is enough to prove that $\cO'$ takes admissible covers in $\cT$ to effective epimorphisms.
	Let $\{U_i \to U\}_{i \in I}$ be an admissible cover in $\cT$.
	Since $\alpha \colon \cO' \to \cO$ is a local transformation, we see that the diagram in $\cX$
	\[ \begin{tikzcd}
		\coprod_{i \in I} \cO'(U_i) \arrow{r} \arrow{d}{\alpha_{U_i}} & \cO'(U) \arrow{d}{\alpha_U} \\
		\coprod_{i \in I} \cO(U_i) \arrow{r} & \cO(U)
	\end{tikzcd} \]
	is a pullback square.
	The conclusion now follows from the fact that effective epimorphisms in $\infty$-topoi are stable under base change (see \cite[6.2.3.15]{HTT}).
\end{proof}

\begin{prop} \label{prop:augmented_local_structures_sifted_colimits}
	Let $\cT$ be a pregeometry and suppose that the topology on $\cT$ is the trivial one.
	Let $\cO \in \Strloc_\cT(\cX)$.
	Then the natural inclusion
	\[ j \colon \Strloc_\cT(\cX)_{/\cO} \to \Str_\cT(\cX)_{/\cO} \]
	commutes with sifted colimits.
\end{prop}

\begin{proof}
	It is enough to prove that the inclusion
	\[ i \colon \Strloc_\cT(\cX)_{/\cO} \hookrightarrow \Fun(\cT, \cX)_{/\cO} \]
	commutes with sifted colimits.
	Let therefore $I$ be a sifted category and let
	\[ F \colon I \to \Strloc_\cT(\cX)_{/\cO} \]
	be a diagram.
	For $i \in I$, set $\cO_i \coloneqq F(i)$.
	Observe that the composition
	\[ \begin{tikzcd} [column sep = small]
		I \arrow{r}{F} & \Strloc_\cT(\cX)_{/\cO} \arrow{r}{i} & \Fun(\cT, \cX)_{/\cO} \arrow{r} & \Fun(\cT \times \Delta^1, \cX) \arrow{r}{\ev_1} & \Fun(\cT, \cX)
	\end{tikzcd} \]
	is the constant diagram associated to $\cO$.
	Since sifted categories are weakly contractible, the $\infty$-categorical version of Quillen's theorem A \cite[4.1.3.1]{HTT} implies that the colimit of this diagram is equivalent to $\cO$.
	Since colimits in $\Fun(\cT \times \Delta^1, \cX)$ are computed objectwise, we conclude that the colimit of the composition
	\[ \begin{tikzcd}[column sep = small]
		I \arrow{r}{F} & \Strloc_\cT(\cX)_{/\cO} \arrow{r}{i} & \Fun(\cT, \cX)_{/\cO} \arrow{r} & \Fun(\cT \times \Delta^1, \cX)
	\end{tikzcd} \]
	lies in $\Fun(\cT, \cX)_{/\cO}$.
	We can therefore represent this colimit as a morphism
	\[ \alpha \colon \cO' \to \cO , \]
	where $\cO' \in \Fun(\cT, \cX)$.
	In order to complete the proof of the claim, it is enough to prove the following two facts:
	\begin{enumerate}[(i)]
		\item the morphism $\alpha$ is local, and
		\item $\cO'$ commutes with finite products and admissible pullbacks.
	\end{enumerate}
	Since sifted colimits commute with finite products, and since each $\cO_i$ is a $\cT$-structure, we conclude that $\cO'$ commutes with finite products.
	Let now $U \to V$ be an admissible morphism in $\cT$.
	Then for every $i \in I$, the diagram
	\[ \begin{tikzcd}
		\cO_i(U) \arrow{r} \arrow{d} & \cO_i(V) \arrow{d} \\
		\cO(U) \arrow{r} & \cO(V)
	\end{tikzcd} \]
	is a pullback square in $\cX$.
	The descent property for $\infty$-topoi guarantees that the the above square remains a pullback after passing to the colimit with respect to $i \in I$.
	In other words, we see that $\alpha$ is a local morphism.
	We are left to prove that $\cO'$ commutes with admissible pullbacks.
	Let
	\[ \begin{tikzcd}
		U' \arrow{r}{f'} \arrow{d} & V' \arrow{d} \\
		U \arrow{r}{f} & V
	\end{tikzcd} \]
	be an admissible pullback in $\cT$.
	In particular, both $f$ and $f'$ are admissible.
	Consider the following commutative cube
	\[ \begin{tikzcd} [column sep = small, row sep = small]
		{} & \cO'(U') \arrow{rr} \arrow{dl} \arrow{dd} & & \cO'(V') \arrow{dl} \arrow{dd} \\
		\cO'(U) \arrow[crossing over]{rr} \arrow{dd} & & \cO'(V) \\
		{} & \cO(U') \arrow{rr} \arrow{dl} & & \cO(V') \arrow{dl} \\
		\cO(U) \arrow{rr} & & \cO(V) . \arrow[leftarrow,crossing over]{uu}
	\end{tikzcd} \]
	Since $\cO$ is a $\cT$-structure, we see that the bottom square is a pullback.
	Moreover, since $\alpha$ is a local morphism and $f$ and $f'$ are admissible, we see that the squares on the sides are pullback as well.
	In particular, the outer and hte right square in the commutative rectangle
	\[ \begin{tikzcd}
		\cO'(U') \arrow{r} \arrow{d} & \cO'(V') \arrow{r} \arrow{d} & \cO(V') \arrow{d} \\
		\cO'(U) \arrow{r} & \cO'(V) \arrow{r} & \cO(V)
	\end{tikzcd} \]
	are pullback squares.
	It follows from \cite[4.4.2.1]{HTT} that the left square is also a pullback.
	As a consequence, $\cO'$ is a $\cT$-structure.	
\end{proof}

\begin{cor} \label{cor:augmented_local_strutures_presentable}
	Let $\cT$ be a pregeometry, $\cX$ an $\infty$-topos and $\cO \in \Strloc_\cT(\cX)$ a $\cT$-structure.
	The $\infty$-category $\Strloc_\cT(\cX)_{/\cO}$ is a presentable $\infty$-category.
\end{cor}

\begin{proof}
	In virtue of \cref{prop:semi_discrete_simplification}, we can replace $\cT$ with the underlying semi-discrete pregeometry $\cT_{\mathrm{sd}}$.
	Equivalently, we can assume from the very beginning that the topology on $\cT$ is the trivial one.
	In this situation, observe that the $\infty$-category $\Str_\cT(\cX)$ coincides with the full subcategory of $\Fun(\cT, \cX)$ spanned by those functors that preserve products and admissible pullbacks.
	It follows from \cite[5.5.4.19]{HTT} that this category is presentable.
	As a consequence, $\Str_{\cT}(\cX)_{/\cO}$ is presentable as well.
	In order to complete the proof, it is enough to show that $\Strloc_{\cT}(\cX)_{/\cO}$ is an accessible localization of $\Str_\cT(\cX)_{/\cO}$.
	
	We start by remarking that the natural functor
	\[ j \colon \Strloc_{\cT}(\cX)_{/\cO} \to \Str_\cT(\cX)_{/\cO} \]
	is fully faithful.
	Reasoning as in the proof of \cref{prop:semi_discrete_simplification}, we see that we can embed both $\Strloc_\cT(\cX)_{/\cO}$ and $\Str_\cT(\cX)_{/\cO}$ as full subcategories of $\Fun_\cO(\cT \times \Delta^1, \cX)$, and that $j$ is compatible with these embeddings.
	In particular, we see that $j$ is fully faithful.
	Furthermore, \cref{prop:augmented_local_structures_sifted_colimits} shows that $j$ commutes with sifted colimits.
	As a consequence, it is enough to show that $j$ admits a left adjoint
	\[ L \colon \Str_\cT(\cX)_{/\cO} \to \Strloc_\cT(\cX)_{/\cO} . \]
	In order to obtain this, it is enough to show that for every object
	\[ \alpha \colon \cO' \longrightarrow \cO \]
	in $\Str_\cT(\cX)_{/\cO}$ there exists a factorization
	\begin{equation} \label{eq:local_factorization}
		\cO' \xrightarrow{\alpha'} \cO'' \xrightarrow{\alpha''} \cO
	\end{equation}
	such that $\alpha''$ is a morphism in $\Strloc_\cT(\cX)$ and that for every $\beta \colon \widetilde{\cO} \to \cO$ in $\Strloc_\cT(\cX)_{/\cO}$, the map induced by $\alpha'$
	\[ \Map_{/\cO}( \cO'', \widetilde{\cO} ) \to \Map_{/\cO}(\cO', \widetilde{\cO}) \]
	is an equivalence.
	
	Recall from \cite[Theorem 1.3.1]{DAG-V} that there exists a factorization system $(S_L, S_R)$ on $\Str_\cT(\cX)$ such that $S_R$ is the collection of morphisms in $\Strloc_\cT(\cX)$.
	Let $\cD$ be the full subcategory of $\Fun(\Delta^1, \Str_\cT(\cX))$ spanned by the elements of $S_R$.
	Using \cite[5.2.8.19]{HTT}, we see that $\cD$ is a localization of $\Fun(\Delta^1, \Str_\cT(\cX))$.
	Denote by $L_{\Delta^1}$ the left adjoint of the inclusion $\cD \hookrightarrow \Fun(\Delta^1, \Str_\cT(\cX))$.
	Then we can identify $L_{\Delta^1}(\alpha)$ with a factorization of the form \eqref{eq:local_factorization}, where $\alpha''$ belongs to $S_R$.
	The universal property of $L_{\Delta^1}$ guarantees that $\alpha' \colon \cO' \to \cO''$ is universal with respect to this kind of factorization.
	This completes the proof.
\end{proof}

\begin{cor} \label{cor:augmented_local_structures_change_of_pregeometry}
	Let $\varphi \colon \cT' \to \cT$ be a transformation of pregeometries.
	Let $\cX$ be an $\infty$-topos and let $\cO \in \Strloc_\cT(\cX)$.
	Then the functor
	\[ \varphi_* \colon \Strloc_\cT(\cX)_{/\cO} \to \Strloc_{\cT'}(\cX)_{/\cO \circ \varphi} \]
	commutes with limits and sifted colimits.
\end{cor}

\begin{proof}
	Using \cref{prop:semi_discrete_simplification}, we can replace both $\cT$ and $\cT'$ with the underlying semi-discrete pregeometries.
	In other words, we can assume the topologies on $\cT$ and on $\cT'$ to be trivial from the very beginning.
	Consider the commutative diagram
	\[ \begin{tikzcd}
		\Strloc_{\cT}(\cX)_{/\cO} \arrow{r}{\varphi_*} \arrow[hook]{d} & \Strloc_{\cT'}(\cX)_{/\cO \circ \varphi} \arrow[hook]{d} \\
		\Fun(\cT, \cX)_{/\cO} \arrow{r}{\varphi_*} & \Fun(\cT', \cX)_{/\cO \circ \varphi} .
	\end{tikzcd} \]
	As we saw in the proof of \cref{prop:semi_discrete_simplification}, the vertical arrows are fully faithful.
	Furthermore, the proof of \cref{prop:augmented_local_structures_sifted_colimits} shows that they commute with sifted colimits.
	Since limits in $\Strloc_\cT(\cX)_{/\cO}$ and in $\Strloc_{\cT'}(\cX)_{/\cO \circ \varphi}$ can be computed objectwise, we conclude that the vertical functors in the above diagrams commute with limits as well.
	
	At this point, it is enough to observe that the functor
	\[ \varphi_* \colon \Fun(\cT, \cX)_{/\cO} \to \Fun(\cT', \cX)_{/\cO \circ \varphi} \]
	commutes with both limits and colimits.
\end{proof}

\subsection{Truncation of structures} \label{subsec:truncation_structures}

If $A$ is a simplicial commutative ring, the truncations $\tau_{\le n} A$ are fundamental objects in derived algebraic geometry.
All together, they form the \emph{Postnikov tower} of $A$, which is often used in order to reduce the proof of derived statements to the corresponding underived ones.

When working with a general pregeometry, it is desirable to have an analogous notion of truncation.
In general, this is not possible: ultimately, the reason is to be found in the fact that the truncation functor in an $\infty$-topos does not commute with finite limits.
For this reason, J.\ Lurie introduced the notion of \emph{being compatible with $n$-truncations}:

\begin{defin}
	Let $\cT$ be a pregeometry and let $n \ge 0$ be an integer.
	We say that $\cT$ is compatible with $n$-truncations if for every $\infty$-topos $\cX$, every $\cT$-structure $\cO \in \Str_\cT(\cX)$ and every admissible morphism $U \to V$ in $\cT$, the square
	\[ \begin{tikzcd}
		\cO(U) \arrow{r} \arrow{d} & \cO(V) \arrow{d} \\
		\tau_{\le n}( \cO(U) ) \arrow{r} & \tau_{\le n}( \cO(V) )
	\end{tikzcd} \]
	is a pullback.
\end{defin}

It is shown in \cite[Proposition 3.3.3]{DAG-V} that if $\cT$ is a pregeometry compatible with $n$-truncations and $\cO \in \Str_\cT(\cX)$ is a $\cT$-structure, then the composition $\tau_{\le n} \circ \cO$ is again a $\cT$-structure.
We denote this $\cT$-structure by $\tau_{\le n} \cO$.

Given a pregeometry $\cT$, it is often useful to consider the full subcategory of $\RTop(\cT)$ spanned by those $\cT$-structured topoi $X \coloneqq (\cX, \cO_X)$ such that $\cO_X$ is $n$-truncated.
We denote this $\infty$-category by $\RTop_{\le n}(\cT)$.
It is useful to provide a different description of this $\infty$-category, based on the notion of geometry.

A \emph{geometry} is a variation of the notion of pregeometry, also introduced in \cite{DAG-V}.
Every pregeometry $\cT$ has an associated universal geometry $\cG$, that can informally be thought as the free completion of $\cT$ under finite limits (see \cite[\S 3.4]{DAG-V}).
It is characterized by the property that, for every $\infty$-topos $\cX$, the canonical functor $\cT \to \cG$ induces an equivalence
\begin{equation} \label{eq:pregeometry_vs_geometry}
	\Str_\cG(\cX) \xrightarrow{\sim} \Str_\cT(\cX) .
\end{equation}

Given a geometry $\cG$, we can construct a new geometry $\cG_{\le n}$ and a morphism of geometries
\[ \cG \longrightarrow \cG_{\le n} \]
inducing for any $\infty$-topos $\cX$ an equivalence
\begin{equation} \label{eq:n_stub}
	\Str_\cG^{\le n}(\cX) \xrightarrow{\sim} \Str_\cG(\cX) ,
\end{equation}
where the left hand side denotes the full subcategory of $\Str_\cG(\cX)$ spanned by $n$-truncated $\cG$-structures.
Following \cite{DAG-V}, we refer to $\cG_{\le n}$ as an \emph{$n$-stub of $\cG$}.
Combining the equivalences \eqref{eq:pregeometry_vs_geometry} and \eqref{eq:n_stub}, we obtain the equivalences
\[ \RTop(\cT) \simeq \RTop(\cG) , \qquad \RTop_{\le n}(\cT) \simeq \RTop(\cG_{\le n}) . \]

We have the following elementary lemma:

\begin{lem} \label{lem:truncation_as_change_of_geometry}
	Let $\cT$ be a pregeometry compatible with $n$-truncations.
	Let $\cG$ be the associated geometry and let $\cG_{\le n}$ be an $n$-stub of $\cG$.
	Then the change of geometry functor of \cite[Theorem 2.1.12]{DAG-V}
	\[ \Spec_\cG^{\cG_{\le n}} \colon \RTop(\cG) \to \RTop(\cG_{\le n}) \]
	sends $X \coloneqq (\cX, \cO_X)$ to $(\cX, \tau_{\le n} \cO_X)$.
\end{lem}

\begin{proof}
	Let $X \coloneqq (\cX, \cO_X)$ be a $\cG$-structured topos.
	It follows from \cite[Proposition 3.3.3]{DAG-V} that $\tau_{\le n} \cO_X$ is a $\cG$-structure on $\cX$.
	We set $\mathrm t_{\le n} X \coloneqq (\cX, \tau_{\le n} \cO_X)$.
	Since $\tau_{\le n} \cO_X$ is $n$-truncated, \cite[Proposition 1.5.14]{DAG-V} shows that it defines a $\cG_{\le n}$-structure on $\cX$.
	The morphism $\cO_X \to \tau_{\le n} \cO_X$ is a local morphism because $\cG$ is compatible with $n$-truncations.
	Therefore, it defines a well defined morphism $p \colon \mathrm t_{\le n} X \to X$ in $\RTop(\cG)$.
	We claim that for every $Y \coloneqq (\cY, \cO_Y) \in \Top(\cG_{\le n})$, the canonical morphism induced by $p$
	\[ \Map_{\RTop(\cG_{\le n})}(Y, \mathrm t_{\le n} X) \to \Map_{\RTop(\cG)}(Y, X) \]
	is a homotopy equivalence.
	Indeed, we have a commutative diagram of fiber sequences
	\[ \begin{tikzcd}
		\Map_{\Strloc_{\cG_{\le n}}(\cY)}(f\inv \tau_{\le n} \cO_X, \cO_Y) \arrow{r} \arrow{d} & \Map_{\Strloc_{\cG}(\cY)}(f\inv \cO_X, \cO_Y ) \arrow{d} \\
		\Map_{\RTop(\cG_{\le n})}(Y,  \mathrm t_{\le n} X) \arrow{r} \arrow{d} & \Map_{\RTop(\cG)}(Y, X) \arrow{d} \\
		\Map_{\RTop}(\cY, \cX) \arrow[equal]{r} & \Map_{\RTop}(\cY, \cX)
	\end{tikzcd} \]
	where both the fibers are computed over the geometric morphism $f\inv \colon \cX \rightleftarrows \cY \colon f_*$.
	Since the left adjoint $f\inv$ is left exact, \cite[5.5.6.28]{HTT} produces the following equivalence:
	\[ f\inv \tau_{\le n} \cO_{\cX} \simeq \tau_{\le n} f\inv \cO_{\cX} . \]
	Finally, since the functor $\Strloc_{\cG_{\le n}}(\cY) \to \Strloc_{\cG}(\cY)$ is fully faithful (in virtue of \cite[Proposition 1.5.14]{DAG-V}), we see that the top horizontal morphism is a homotopy equivalence.
	The proof is therefore complete.
\end{proof}

\begin{prop} \label{prop:relative_spectrum_and_truncations}
	Let $\varphi \colon \cG' \to \cG$ be a transformation of geometries and suppose that both $\cG'$ and $\cG$ are compatible with $n$-truncations.
	Let $X \coloneqq (\cX, \cO_X) \in \Top(\cG')$.
	Then the canonical morphism
	\[ \mathrm{Spec}^{\cG}_{\cG'}(\cX, \tau_{\le n} \cO_X) \to \mathrm{Spec}^{\cG}_{\cG'}(\cX, \cO_X) \]
	Exhibits $\mathrm{Spec}^{\cG}_{\cG'}(\cX, \tau_{\le n} \cO_X)$ as $n$-truncation of $\mathrm{Spec}^{\cG}_{\cG'}(\cX, \cO_X)$.
	In particular, it induces an equivalence on the underlying $\infty$-topos.
\end{prop}

\begin{proof}
	Let $\cG' \to \cG'_{\le n}$ and $\cG \to \cG_{\le n}$ be $n$-stubs for $\cG'$ and $\cG$ respectively.
	The universal property defining $n$-stubs, implies the existence of a commutative square of morphism of geometries
	\[ \begin{tikzcd}
	\cG' \arrow{r}{\varphi} \arrow{d} & \cG \arrow{d} \\
	\cG'_{\le n} \arrow{r}{\varphi_n} & \cG_{\le n}
	\end{tikzcd} \]
	Therefore we have
	\[ \mathrm{Spec}^{\cG_{\le n}}_{\cG} \circ \mathrm{Spec}^{\cG}_{\cG'} \simeq \mathrm{Spec}^{\cG_{\le n}}_{\cG'_{\le n}} \circ \mathrm{Spec}^{\cG'_{\le n}}_{\cG'} \]
	Combining this with \cref{lem:truncation_as_change_of_geometry}, we obtain the desired result.
\end{proof}

Recall from \cite[\S 2.3]{DAG-V} the notion of $\cG$-scheme for a geometry $\cG$.
We have:

\begin{prop} \label{prop:relative_spectrum_truncated_objects}
	Let $\cG, \cG'$ be geometries compatible with $n$-truncations.
	Let $\varphi \colon \cG' \to \cG$ be a transformation of geometries and let $\cG' \to \cG'_{\le n}$, $\cG \to \cG_{\le n}$ be $n$-stubs for $\cG'$ and $\cG$, respectively.
	The diagram
	\[ \begin{tikzcd}
	\Sch(\cG'_{\le n}) \arrow{rr}{\mathrm{Spec}^{\cG_{\le n}}_{\cG'_{\le n}}} \arrow{d} & & \Sch(\cG_{\le n}) \arrow{d} \\
	\Sch(\cG') \arrow{rr}{\mathrm{Spec}^{\cG}_{\cG'}} & & \Sch(\cG)
	\end{tikzcd} \]
	commutes.
\end{prop}

\begin{proof}
	Since $\varphi$ commutes with finite limits, the induced morphism
	\[ \widetilde{\varphi} \colon \mathrm{Ind}((\cG')^{\mathrm{op}}) \to \mathrm{Ind}(\cG^{\mathrm{op}}) \]
	commutes with finite limits as well.
	In particular, we see that it takes $k$-truncated objects to $k$-truncated objects.
	The statement now follows from the following pair of observations:
	\begin{enumerate}
		\item if $A \in \mathrm{Ind}(\cG^{\mathrm{op}})$ is $k$-truncated, then, writing $\Spec^{\cG}(A) = (\cX_A, \cO_A)$, $\cO_A$ is $k$-truncated;
		\item if $A \in \mathrm{Ind}((\cG')^{\mathrm{op}})$, then $\mathrm{Spec}^{\cG}_{\cG'}(\Spec^{\cG}(A)) \simeq \mathrm{Spec}^{\cG'}(\widetilde{\varphi}(A))$.
	\end{enumerate}
	As these statements follow directly from the definitions, the proof is complete.
\end{proof}

\bibliographystyle{plain}
\bibliography{dahema}

\end{document}